\documentclass[11pt,a4paper]{article}
\usepackage[utf8]{inputenc}
\usepackage{lmodern}
\usepackage[T1]{fontenc}
\usepackage[english]{babel}
\usepackage{ifpdf}
\usepackage[left=1in, right=1in, top=1in, bottom=1in]{geometry}
\usepackage[dvipsnames]{xcolor}
\usepackage[colorlinks=true,linkcolor=blue,citecolor=green]{hyperref}

\usepackage{graphicx, amsmath, amsthm, amssymb, enumitem, mathrsfs} 
\usepackage{subcaption}
\usepackage{booktabs} 

\usepackage[capitalize,nameinlink]{cleveref}
\usepackage[square,numbers]{natbib}
\usepackage{soul}

\usepackage{tikz}
\usepackage{tikz-3dplot} 
\usepackage{pgfplots, pgfplotstable}
\usetikzlibrary{3d} 
\usetikzlibrary{arrows.meta}

\usepackage{algorithm}
\usepackage{algorithmic}

\newcommand{\ep}[1]{\mathbb{E}\left({#1}\right)}

\newcommand{\distmatrix}{D}
\newcommand{\x}{\vec{x}}
\newcommand{\y}{\vec{y}}
\newcommand{\one}{\vec{1}}
\newcommand{\zero}{\vec{0}}

\newcommand{\Path}{\mathsf{P}}
\newcommand{\Basket}{\mathsf{B}}
\newcommand{\Cube}{\mathsf{Q}}
\newcommand{\Cycle}{\mathsf{C}}
\newcommand{\Complete}{\mathsf{K}}

\newtheorem{theorem}{Theorem}[section]

\newtheorem{remark}[theorem]{Remark}
\newtheorem{definition}[theorem]{Definition}
\newtheorem{lemma}[theorem]{Lemma}
\newtheorem{proposition}[theorem]{Proposition}

\newtheorem{example}[theorem]{Example}
\newtheorem{observation}[theorem]{Observation}

\usepackage{xcolor, soul}

\newcommand{\Circ}{\circ}

\title{Distance Exceptional Graphs\\and the Curvature Index}
\author{Sawyer Jack Robertson\thanks{Department of Mathematics, UC San Diego\\\hspace*{0.5cm}\texttt{MSC2020:} 05C50, 05C12, 05C76 \\\hspace*{0.5cm}\texttt{Keywords:} graph distance matrix, discrete curvature, graph Cartesian product, graph join}\and Finn Southerland\footnotemark[1]\and Erlang Surya\footnotemark[1]}
\date{}

\hypersetup{
	pdftitle={Distance Exceptional Graphs and the Curvature Index},
	pdfauthor={S. J. Robertson, F. Southerland, E. Surya}
} 

\begin{document}
	\captionsetup[figure]{labelfont={bf},labelformat={default},labelsep=period,name={Figure}}
	\maketitle

    \begin{abstract}
        A graph $G=(V,E)$ on $n$ vertices is said to be \emph{distance exceptional} if the equation $D\vec{x} = \vec{1}$ admits no solution $\vec{x}\in\mathbb{R}^{n}$, where $D\in\mathbb{R}^{n\times n}$ is the shortest path distance matrix of $G$. These graphs were first studied by Steinerberger in the context of a notion of discrete curvature (``Curvature on graphs via equilibrium measures,'' \emph{Journal of Graph Theory}, 103(3), 2023). This work has led to several open questions about distance exceptional graphs, including: What is the structure of such graphs? How can they be characterized? How rare are they? In this paper, we investigate these questions through the lens of a graph invariant we term the \emph{curvature index}. We show that a graph is distance exceptional if and only if this invariant vanishes, and we develop a calculus for this invariant under graph operations including the Cartesian product and graph join. As a result, we recover and generalize a number of known results in this area. We show that any graph $G$ can be realized as an induced subgraph of a distance exceptional graph $G'$. Moreover, in many cases, this embedding is an isometry. In turn, this leads to a number of methods for constructing distance exceptional graphs. 
    \end{abstract}


    \section{Introduction}\label{sec:introduction}

    Let $G=(V,E)$ be a graph, where $V=V(G)=\{v_1,\dotsc, v_n\}$ is a fixed vertex set and $E=E(G)\subseteq \binom{V}{2}$ is a collection of $m\geq 1$ undirected edges. We assume that $G$ is connected. For $x, y\in V$, we denote by $d(x)$ the vertex degree of $x$ and by $d(x, y)$ the shortest path distance between $x$ and $y$. The matrix $\distmatrix\in\mathbb{R}^{n\times n}$ with entries $\distmatrix=(d(v_i,v_j))_{i, j=1}^{n}$ is known as the distance matrix of $G$. We denote by $\one_k\in\mathbb{R}^k$ the $k$-vector of all ones, and we suppress the subscript $k$ by writing simply $\one$ when the ambient dimension of the vector can be understood from surrounding context.

    In 2023, Steinerberger introduced a notion of curvature on graphs by considering the equation $\distmatrix\x  = n\one$ in the unknown $\x $, whose solution (or an approximation thereof) was used to define the curvature at each vertex~\cite{steinerberger2023curvature}. The author showed that this notion of graph curvature satisfies several desirable properties, including a Bonnet--Myers-type diameter bound (see~\cite[Thm. 1]{steinerberger2023curvature}) and a Lichnerowicz-type theorem relating curvature to graph Laplacian eigenvalues (see~\cite[Thm. 2]{steinerberger2023curvature}). Subsequent work has explored closed-form computation of this curvature vector in various settings, such as block graphs~\cite{cushing2024note}. 
    
    However, while the existence of a solution to this system is not always guaranteed, the system $\distmatrix\x = \one$ seems to admit a solution intriguingly often. In settings where $\distmatrix\x  = \one$ does \emph{not} admit a solution, we say the graph $G$ is \emph{distance exceptional}. In~\cite{steinerberger2023curvature}, the author points out, for example, that in a \emph{Mathematica} database of 9,059 graphs satisfying $2\leq n\leq 500$, only five fail to admit a solution to the equation $\distmatrix\x  = \one$. This has led to several open questions about distance exceptional graphs, including: What is the structure of such graphs? How can they be characterized? How rare are they? 

    In later work~\cite{dudarov2023image}, the authors prove that, despite the apparent rarity, such graphs \emph{do} occur for every order $n\ge 7$. They also show that for Erd\H{o}s--R\'enyi random graphs, the distance matrix is invertible, and thus $\distmatrix\x =\one$ has a unique solution, asymptotically almost surely as $n\rightarrow\infty$. This reinforces the view that distance exceptional graphs are rare in practice. We illustrate six examples of distance exceptional graphs in~\cref{fig:small-dx-examples}.

    \begin{figure}[t!]
        \centering
        \input{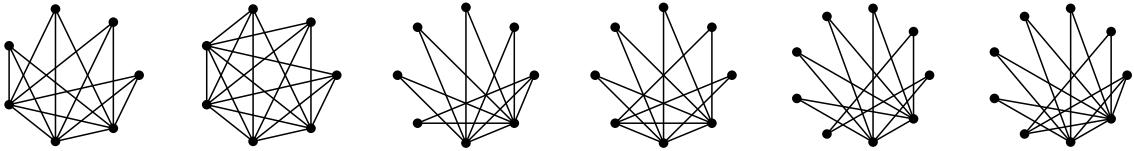}
        \caption{This figure contains illustrations of six distance exceptional graphs. From left to right: the two distance exceptional graphs on seven vertices, followed by two examples each on eight and nine vertices, respectively.}\label{fig:small-dx-examples}
    \end{figure}

    Structurally, the smallest nontrivial distance exceptional graphs appear on $n=7$ vertices (up to isomorphism, there are two), and the counts for $n=8,9,$ and $10$ are $14$, $398$, and $23,923$, respectively~\cite{OEISA354465}. Despite their rarity, this family of graphs demonstrates certain closure properties, in the following two senses. First, if $G, H$ are distance exceptional, then their \emph{graph Cartesian product} $G\,\Box\, H$, formed by connecting vertices $(u, v)$ and $(u', v')$ belonging to $V(G)\times V(H)$ if either $u=u'$ and $\{v, v'\}\in E(H)$ or $v=v'$ and $\{u, u'\}\in E(G)$, is also distance exceptional~\cite[Thm. 6]{dudarov2023image}. Second, if $G, H$ are distance exceptional and $u\in V(G)$, $v\in V(H)$ are fixed, then the \emph{vertex coalescence} $G\,\Circ_{u,v}\,H$ obtained by identifying $u$ with $v$ and thereby connecting $G$ to $H$, is again distance exceptional~\cite[Thm. 6]{chen2023steinerberger}. Moreover, distance exceptional graphs can be constructed from graphs $G$ and $H$ satisfying certain properties (omitted here) by means of the \emph{graph join} $G+ H$ obtained by connecting every vertex of $G$ to every vertex of $H$~\cite[Thm. 4]{dudarov2023image}. These closure properties suggest that the family of distance exceptional graphs, while possibly thin, admits some internal structure.

    These closure results have been used to obtain constructions of distance exceptional graphs of arbitrarily large sizes by taking known examples and then gluing, multiplying, or joining them together in a structured fashion. Despite this, it is not clear based on this body of work how exactly these pathological ``building blocks'' (if we may term them as such) appear in the first place. 

    Our approach to these questions was initially inspired by the following empirical observation, which will serve as a motivating example for our results that follow.
    
    \begin{observation}\label{obs:jailbroken-baskets}
        Each graph depicted below has thirteen vertices and fifteen edges, and is distance exceptional.
        
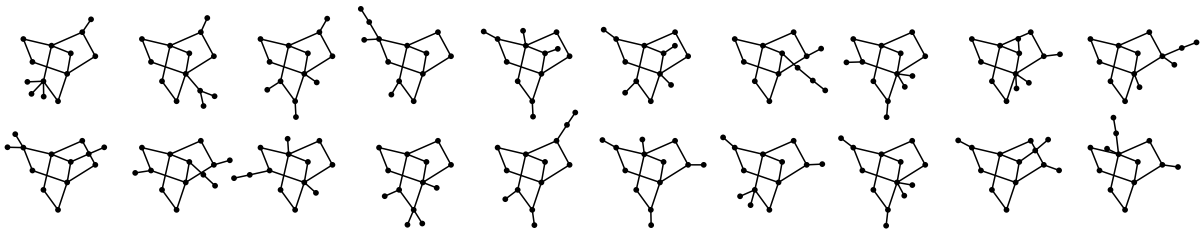
{
            \begin{center}
                \input{figures/basket_dx.tex}
            \end{center}
        }
        Moreover, thousands of additional distance exceptional examples can be found in the same manner: starting from the common induced subgraph consisting of two intersecting cycles, iteratively adjoin exactly four new pendant vertices. Adding greater or fewer than four pendant vertices results in a graph that is not distance exceptional.
    \end{observation}

    The graphs shown in~\cref{obs:jailbroken-baskets} are perhaps surprising since conventional wisdom suggests that distance exceptional graphs should be exceedingly rare and highly constrained. Despite this, the above collection apparently demonstrates certain degrees of freedom. In fact, this construction can be generalized by defining the following family of so-called \emph{basket graphs}, which serve to capture the common induced subgraph shown in~\cref{obs:jailbroken-baskets}.

    \begin{definition}[Basket graphs]\label{def:basket}
        Let $k\geq 1$. Letting $\Path_k$ denote the path graph on $k$ vertices, start with disjoint copies of the graphs $\Path_{k}, \Path_{k+1}, \Path_{k+1},$ and $\Path_{k+1}$. Then join all four paths together at their endpoints, creating two vertices of degree four. The resulting graph is called the $k$-th \emph{basket graph}, denoted $\Basket_k$.
    \end{definition}

    Basket graphs are illustrated in~\cref{fig:basket}. The following theorem makes~\cref{obs:jailbroken-baskets} precise.

    \begin{theorem}\label{thm:basket-jailbreak-intro}
        Let $j\geq 1$ be fixed. Let
            \begin{align}\label{eq:basket-iota}
                s = \frac{ (-3)^{2j+1} + 4(2j+1)-1}{8} \in\mathbb{Z}_{< 0}.
            \end{align}
        Let $G_0$ be the basket graph $\Basket_{2j+1}$ as in~\cref{def:basket}. For $i\ge 1$, let $G_{i}$ be obtained by adding a pendant to the graph $G_{i-1}$ in any manner. Let $G = G_{2|s|}$ be \emph{any} graph obtained from this procedure after exactly $2|s|$ steps. Then $G$ is distance exceptional.
    \end{theorem}
    
    \Cref{thm:basket-jailbreak-intro} is proved from first principles in~\cref{sec:basket-example} (see also~\cref{rmk:basket}). This procedure of starting with a graph $G$ which is \emph{not} distance exceptional and modifying it through the addition of vertices and edges in such a way as to obtain a distance exceptional graph can be generalized to include other graph operations and modifications. As these procedures are carried out, we observe that the proximity of a given graph to being distance exceptional can be captured in some sense by a graph invariant which we term the \emph{curvature index}. This invariant can be defined in several ways, including by using the Moore--Penrose inverse of $\distmatrix$, or the spectral decomposition of $\distmatrix$ (see~\cref{prop:representations-of-index}). We present here the version which is used most often in this article and which is geometric in nature. If $G$ is a graph and $\x\in\mathbb{R}^{n}$ is such that $\distmatrix\x = \one$, we say $\x$ is a \emph{curvature potential} of $G$.

    \begin{definition}[Curvature index]\label{defn:curvature-index}
        Let $G=(V,E)$ be a connected graph. Assume either that there exists a curvature potential $\x\in\mathbb{R}^{n}$ of $G$ such that $\one^\top\x \ne 0$ or that $G$ is distance exceptional. Then the \emph{curvature index} of $G$ is the unique real number $\iota(G)$ such that the affine space
            \begin{align*}
                X(G) &:= \left\{\distmatrix\x \;:\; \x\in\mathbb{R}^{n},\, \x^\top\one = 1\right\}
            \end{align*}
        satisfies $X(G)\cap\mathbb{R}\one = \iota(G)\one$. If all curvature potentials of $G$ have zero mean, then $\iota(G) := \infty$.
    \end{definition}
    
    That $X(G)\cap\mathbb{R}\one$ consists of a single point is a nontrivial fact for which we present a proof in~\cref{sec:curvature-index} (see~\cref{prop:index-of-curvature}). It is straightforward to show that $\iota(G) = 0$ if and only if $G$ is distance exceptional. As a result, by developing a calculus for this invariant under graph operations, we can recover and generalize many of the known results about distance exceptional graphs, in addition to providing a variety of means for constructing examples of these graphs with desired properties. Below we state several of our main results.

    \begin{theorem}[See \cref{thm:index-product}, \cref{thm:merging}]\label{thm:intro-calculus}
        Let $G, H$ be connected graphs with $\iota(G),\iota(H) < \infty$. Then
            \begin{enumerate}
                \item $\iota(G\,\Box\,H) = \iota(G) + \iota(H)$,
                \item For each $u\in V(G), v\in V(H)$, $\iota(G\,\Circ_{u,v}\, H) = \iota(G) + \iota(H)$.
            \end{enumerate}
        If either $\iota(G) = \infty$ or $\iota(H) = \infty$, then $\iota(G\,\Box\,H) = \iota(G\,\Circ_{u,v}\, H) = \infty$.
    \end{theorem}

    \Cref{thm:intro-calculus} is proved in~\cref{sec:curvature-index}. We also consider the case of $G+ H$ in~\cref{thm:index-join}, which is more complicated. Together these results generalize the closure properties of distance exceptional graphs mentioned earlier. We also compute a number of explicit examples of curvature indices for well known graph families, including cycles, hypercubes, trees, and complete multipartite graphs (see~\crefrange{thm:transitive-tree}{thm:kmn}).

    \begin{theorem}\label{thm:intro-rational-index}
        Let $q\in\mathbb{Q}$. Then there exists a graph $G$ such that $\iota(G) = q$.
    \end{theorem}

    \Cref{thm:intro-rational-index} is proved in~\cref{sec:embeddings}, and the proof we offer is constructive, utilizing Egyptian fractions along with graph products and vertex coalescences to exhibit a graph $G$ whose curvature index achieves any given rational number. 

    \begin{theorem}\label{thm:intro-embedding}
        Let $G$ be a graph, connected or otherwise. Then there exists a distance exceptional graph $G'$ containing a copy of $G$ as an induced subgraph. If $G$ is connected and $\iota(G) < \infty$, then this embedding can be taken to be isometric. 
    \end{theorem}

    \Cref{thm:intro-embedding} is proved in~\cref{sec:embeddings}, and is supplemented by~\cref{alg:X}, which describes in an algorithmic fashion how to construct such embeddings in practice. We follow this with two explicit examples.
    
    With these results in hand, we remark that our motivational example~\cref{thm:basket-jailbreak-intro} can be recast as follows. The curvature index of $\Basket_{2j+1}$ is given by~\cref{eq:basket-iota}, and the addition of a pendant has the effect of a coalescence with a copy of $\Complete_2$, which increases the index by $\frac{1}{2}$. Repeating this $2|s|$ times in an otherwise arbitrary manner results in a distance exceptional graph.
    
    Beyond the work presented herein, there are many promising directions for future research. For instance, these authors feel it would be interesting to investigate more deeply the structure of graphs with infinite curvature index. It would also be interesting to explore further applications of the curvature index in constructing distance exceptional graphs with additional properties, such as planarity or bounded degree.

    We remark additionally that this framework leads to natural conjectures on the behavior of the curvature index in the setting of Erd\H{o}s--R\'{e}nyi random graphs. In the regime where $p\in(0,1)$ and $n\to\infty$, it is not hard to prove that $\mathrm{diam}(G)=2$ with probability at least $1-o(1)$, and moreover, the tail probabilities for this event are quite forgiving. Conditional on $\mathrm{diam}(G)=2$, it follows that $\distmatrix(G) = 2(J - I) - A$, so that in expectation,
        \begin{align*}
            \ep{2(J - I) - A} &= 2(J - I) - p(J - I) = (2 - p)(J - I).
        \end{align*}
    This matrix satisfies
        \begin{align*}
            \one^\top \ep{2(J - I) - A}^+ \one &= \frac{1}{2-p}(1+o(1)).
        \end{align*}
    Based on~\cref{prop:representations-of-index}, it is therefore reasonable to conjecture that in the Erd\H{o}s--R\'{e}nyi regime described above, the curvature index of a random graph concentrates near $2-p$ with high probability. Note that any quantitative resolution of this conjecture could conceivably admit, in the setting $p=1/2$, some insight into the fraction of labeled graphs which are distance exceptional. This problem remains open.

    This paper is organized as follows. In~\cref{sec:basket-example}, we consider the setting of~\cref{obs:jailbroken-baskets} more closely and offer an explanation for this phenomenon at a low level to better motivate the subsequent sections. In~\cref{sec:curvature-index}, we rigorously define the curvature index and investigate its behavior under various graph operations. In~\cref{sec:computing-index}, we compute the value of this invariant across a number of graph families of interest. Finally, in~\cref{sec:embeddings}, we prove~\cref{thm:intro-rational-index} and~\cref{thm:intro-embedding}, and exhibit examples of the embeddings mentioned therein. 

    \section{Basket graphs and adding pendants}\label{sec:basket-example}

    In this section we give context for our main results and for the constructions that use them by taking a closer look at the family of distance exceptional graphs highlighted in~\cref{obs:jailbroken-baskets}. These examples arise from a common template we call the \emph{basket graphs}, defined in the previous section (see~\cref{def:basket}), and illustrated in~\cref{fig:basket}.

    \begin{figure}[t!]
        \centering
        \begin{tikzpicture}[x=1cm,y=1cm,
  vtx/.style={circle, draw, fill=black, inner sep=2pt, minimum size=6pt},
  every node/.style={font=\footnotesize}
]

\begin{scope}[shift={(0,0)}]
  \node[vtx,label=left:$v_1$] (L1) at (0,0) {};
  \node[vtx,label=right:$v_3$] (R1) at (5,0) {};

  \node[vtx,label=above:$v_2$] (t) at (2.5,1.8) {};
  \draw[thick] (L1) -- (t) -- (R1);

  \node[vtx,label=above:$v_4$] (a1) at (1.67,0.6) {};
  \node[vtx,label=above:$v_5$] (a2) at (3.33,0.6) {};
  \draw[thick] (L1) -- (a1) -- (a2) -- (R1);

  \node[vtx,label=above:$v_6$] (b1) at (1.67,-0.2) {};
  \node[vtx,label=above:$v_7$] (b2) at (3.33,-0.2) {};
  \draw[thick] (L1) -- (b1) -- (b2) -- (R1);

  \node[vtx,label=above:$v_8$] (c1) at (1.67,-1.0) {};
  \node[vtx,label=above:$v_9$] (c2) at (3.33,-1.0) {};
  \draw[thick] (L1) -- (c1) -- (c2) -- (R1);

  \node[font=\normalsize] at (2.5,2.6) {$\mathcal{B}_3$};
\end{scope}

\begin{scope}[shift={(8,0)}]
  \node[vtx,label=left:$v_1$] (L2) at (0,0) {};
  \node[vtx,label=right:$v_{k}$] (R2) at (5,0) {};

  \node[vtx,label=above:$v_2$] (u2) at (1.0,1.6) {};
  \node (udots) at (2.5,1.6) {$\cdots$};
  \node[vtx,label=above:$v_{k-1}$] (ukm1) at (4.0,1.6) {};
  \draw[thick] (L2) -- (u2) -- (udots) -- (ukm1) -- (R2);

  \node[vtx, label=above:$v_{k+1}$] (a1) at (1.2,0.6) {};
  \node (adots) at (2.5,0.6) {$\cdots$};
  \node[vtx, label=above:$v_{2k-1}$] (a2) at (3.8,0.6) {};
  \draw[thick] (L2) -- (a1) -- (adots) -- (a2) -- (R2);

  \node[vtx, label=above:$v_{2k}$] (b1) at (1.2,-0.2) {};
  \node (bdots) at (2.5,-0.2) {$\cdots$};
  \node[vtx, label=above:$v_{3k-2}$] (b2) at (3.8,-0.2) {};
  \draw[thick] (L2) -- (b1) -- (bdots) -- (b2) -- (R2);

  \node[vtx, label=above:$v_{3k-1}$] (c1) at (1.2,-1.0) {};
  \node (cdots) at (2.5,-1.0) {$\cdots$};
  \node[vtx, label=above:$v_{4k-3}$] (c2) at (3.8,-1.0) {};
  \draw[thick] (L2) -- (c1) -- (cdots) -- (c2) -- (R2);

  \node[font=\normalsize] at (2.5,2.6) {$\mathcal{B}_k$};
\end{scope}

\end{tikzpicture}
        \caption{Two illustrations of the basket graphs $\Basket_k$ with their conventional vertex enumerations. Left: $\Basket_3$. Right: $\Basket_k$. The name ``basket'' reflects the slightly shorter upper path serving as a ``handle'' above the three lower paths.}\label{fig:basket}
    \end{figure}
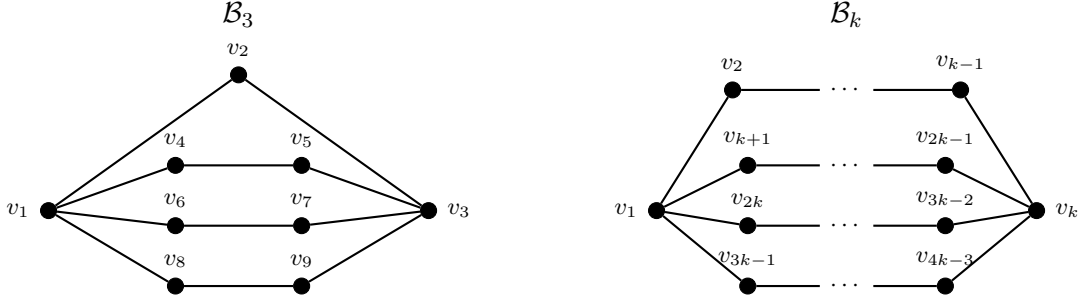

    Since $\distmatrix$ is symmetric, the linear system $\distmatrix\x=\one$ has a solution if and only if $\one\in\mathrm{range}(\distmatrix)$, i.e., if and only if $\mathbb{R}\one\perp \mathrm{ker}(\distmatrix)$. Equivalently:

    \begin{lemma}\label{lem:dne-positive-sum-kernel}
        $\distmatrix\x  = \one $ does not admit a solution if and only if there exists $\x \in\mathbb{R}^{n}$ such that $\one^\top\x  > 0$ and $\distmatrix\x  = \zero$. 
    \end{lemma}

    \Cref{lem:dne-positive-sum-kernel} can be exploited to formalize the construction of graphs such as those shown in~\cref{obs:jailbroken-baskets} as follows. Initially, we establish the existence of a vector satisfying $\one^\top\x  = 1$ and $\distmatrix\x  = r \one $ for some $r\in\mathbb{Z}$ with $r < 0$. We then modify the graph in such a way as to obtain a new vector $\x'$ satisfying \emph{(i)} $\one^\top\x' = 1$, and \emph{(ii)}, $\distmatrix'\x' = (r+1/2) \one $, where $\distmatrix'$ is the distance matrix of the modified graph. By repeating this $2|r|$ times, we obtain a distance exceptional graph via~\cref{lem:dne-positive-sum-kernel}. The following lemma is the first step in this process.

    \begin{lemma}\label{lem:curvature-basket}
        Let $k\geq 3$ and set $n := |V(\Basket_k)| = 4k - 3$. Then there exists a vector $\x \in\mathbb{R}^{n}$ satisfying
            \begin{align}\label{eq:alpha-k}
                \one^\top\x  = 1, \quad \distmatrix\x  = \frac{(-3)^k+4k-1}{8} \one .
            \end{align}
    \end{lemma}

    The proof of~\cref{lem:curvature-basket} is computational and is deferred to~\cref{sec:appendix-basket} in order not to interrupt the flow of this article. The key idea is to exploit the symmetry of $\Basket_k$ to reduce the system $\distmatrix\x  = r \one $ to a smaller system of linear equations, which can then be solved by hand. In the language of curvature indices introduced in~\cref{defn:curvature-index},~\cref{lem:curvature-basket} asserts that $\iota(\Basket_k) =  \frac{(-3)^k+4k-1}{8}$. The second ingredient to our construction is given below.

    \begin{lemma}\label{lem:second-ingredient}
        Let $G=(V, E)$ be a connected graph on $n$ vertices. Assume there exists $\x \in\mathbb{R}^{n}$ such that $\one^\top \x  = 1$ and $\distmatrix\x  = r\one $ for some $r\in\mathbb{R}$. Let $u^\ast\in V$ be a fixed vertex, and let $\widetilde{G}$ be the graph obtained by adding a single vertex $u_{n+1}$ and edge $\{u^\ast, u_{n+1}\}$. Define $\x'\in\mathbb{R}^{n+1}$ entrywise by setting
            \begin{align*}
                \x'_i &=\begin{cases}
                    \x_i - \frac{1}{2} &\text{ if }u_i = u^\ast\\
                    \frac{1}{2} &\text{ if }u_i = u_{n+1}\\
                    \x_i&\text{ otherwise}
                \end{cases},\quad  1\leq i \leq n+1.
            \end{align*}
        Let $\distmatrix'$ denote the distance matrix of $\widetilde{G}$. Then: \emph{(i)} $\one^\top\x' = 1$, and \emph{(ii)}, $\distmatrix'\x'  = (r + \frac{1}{2}) \one $.
    \end{lemma}

    \begin{proof}
        Claim \textit{(i)} is obvious. For claim \textit{(ii)}, we assume without loss of generality that $u^\ast = u_1$. We partition $\distmatrix$ into the following block matrix structure
            \begin{align*}
                \distmatrix &= \begin{bmatrix}
                    0 & \y^\top \\
                    \y  & \distmatrix_1
                \end{bmatrix},
            \end{align*}
        where $\y \in\mathbb{R}^{n-1}$ contains the distances between $u_1$ and the remaining vertices $V\setminus\{u_1\}$, and $\distmatrix_1\in\mathbb{R}^{(n-1)\times (n-1)}$ is the submatrix of $\distmatrix$ indexed by the vertices $V\setminus\{u_1\}$. By writing
            \begin{align*}
                {\x}'' = \begin{bmatrix}
                    x_2 & x_3 &\dotsc x_{n}
                \end{bmatrix}^\top,
            \end{align*}
        we thus have that
            \begin{align}\label{eq:dx}
                \distmatrix\x  &= \begin{bmatrix}
                    0 & \y^\top \\
                    \y  & \distmatrix_1
                \end{bmatrix}\begin{bmatrix}
                    x_1 \\ 
                    {\x}''
                \end{bmatrix} = \begin{bmatrix}
                    \y^\top {\x}'' \\
                    x_1 \y  + \distmatrix_1 {\x}''
                \end{bmatrix} = \begin{bmatrix}
                    r \\
                    r\one
                \end{bmatrix}.
            \end{align}
        By the assumption on $u_1$, we have that for each $v\in V\setminus\{u_1, u_{n+1}\}$, it holds
            \begin{align*}
                d(u_{n+1}, v) &= d(u_{1}, v) + 1.
            \end{align*}
        Therefore we may write the matrix $\distmatrix'$ in the form
            \begin{align*}
                \distmatrix' &= \begin{bmatrix}
                    0 & \y^\top & 1\\
                    \y  & \distmatrix_1 & \y +\one\\
                    1 & (\y  +\one)^\top & 0
                \end{bmatrix}.
            \end{align*}
        By \cref{eq:dx}, it holds
            \begin{align*}
                \distmatrix'\x' &= \begin{bmatrix}
                    0 & \y^\top & 1\\
                    \y  & \distmatrix_1 & \y +\one\\
                    1 & (\y  +\one)^\top & 0
                \end{bmatrix}\begin{bmatrix}
                    x_1 - \tfrac{1}{2}\\
                    {\x}'' \\
                    \tfrac{1}{2}
                \end{bmatrix} = \begin{bmatrix}
                    \y^\top \x'' + \tfrac{1}{2}\\
                    (x_1-\tfrac{1}{2}) \y  + \distmatrix_1\x'' + (\y +\one)\tfrac{1}{2}\\
                    (x_1-\tfrac{1}{2}) + (\y +\one)^\top \x''
                \end{bmatrix}\\
                &= \begin{bmatrix}
                    r + \tfrac{1}{2}\\
                    (r+\tfrac{1}{2})\one\\
                    \one^\top \x  - \frac{1}{2} + r
                \end{bmatrix} = (\tfrac{1}{2}+r)\one.
            \end{align*}
        Thus \emph{(ii)} is proved.
    \end{proof}

    Based on~\cref{lem:dne-positive-sum-kernel} and using~\cref{lem:curvature-basket} and~\cref{lem:second-ingredient},~\cref{thm:basket-jailbreak-intro} immediately follows.
    
    \begin{remark}\label{rmk:basket}
        \Cref{thm:basket-jailbreak-intro} illustrates the mechanism allowing graphs such as those in~\cref{obs:jailbroken-baskets} to emerge, demonstrating several degrees of freedom in their constructions despite the highly constrained nature of distance exceptionality. In fact, the constant $r\in\mathbb{R}$ appearing in~\cref{lem:second-ingredient} is the curvature index, and the process described is vertex coalescence with a copy of $\Complete_2$. Thus ~\cref{lem:second-ingredient} is a special case of \cref{thm:merging}, where $\iota(\Complete_2) = \frac{1}{2}$. It is not clear why the constant $r$ is exponentially large for the family of basket graphs in particular; we note that, empirically, such behavior seems unusual, and is not well understood at present.
    \end{remark}

    \section{The curvature index and graph operations}\label{sec:curvature-index}

    In this section, we formalize the curvature index for a connected graph, and examine its behavior under the graph operations of interest in this article: Cartesian product, vertex coalescence, and graph join. The following proposition, to which we alluded in the introduction, asserts a geometric constraint on the images of unit sum curvature potentials.
    
    \begin{proposition}\label{prop:index-of-curvature}
        Let $G$ be a connected graph on $n$ vertices. Assume either that there exists a curvature potential $\x\in\mathbb{R}^{n}$ of $G$ such that $\one^\top\x \ne 0$ or that $G$ is distance exceptional. Then the affine space
            \begin{align*}
                X(G) &:= \left\{\distmatrix\x \;:\; \x\in\mathbb{R}^{n},\, \x^\top\one = 1\right\}
            \end{align*}
        intersects the line $\mathbb{R}\one$ at exactly one point.
    \end{proposition}

    As in~\cref{defn:curvature-index}, if $G$ is such that the hypotheses of~\cref{prop:index-of-curvature} hold, then the unique value $\iota(G)\in\mathbb{R}$ for which $X(G) \cap \mathbb{R}\one = \iota(G)\one$ is called the \emph{curvature index} of $G$ (see~\cref{defn:curvature-index}). Otherwise if this is not the case, we set $\iota(G) := \infty$ and note that this convention will prove to be mathematically natural, as our statements concerning $\iota(G)$ will hold in this case using reasonable conventions for computations in the extended reals.

    \begin{proof}[Proof of~\cref{prop:index-of-curvature}]
        First we claim that this intersection is not empty. If $G$ is not distance exceptional then there is some $\x\in\mathbb{R}^{n}$ with $\x\neq 0$ such that $\distmatrix\x = \one$. By assumption, this $\x$ may be taken so that $\one^\top\x \neq 0$. We can set $\x_0 := \x / \one^\top\x$ so that $\x_0^\top\one = 1$ and $\distmatrix\x_0\in \mathbb{R}\one$. Otherwise, if $G$ is distance exceptional, by~\cref{lem:dne-positive-sum-kernel}, we may fix some $\y\in\mathbb{R}^{n}$ with $\y^\top\one  \neq0$ and $\distmatrix\y = 0\one\in\mathbb{R}\one$. Again by scaling we can assume $\one^\top \y = 1$. Thus $X(G)\cap\mathbb{R}\one \neq\varnothing$ holds. Next, assume that there are $\x_1,\x_2\in\mathbb{R}^{n}$ for which 
            \begin{align*}
                \distmatrix\x_i = \iota_i\one,\quad \iota_i\in\mathbb{R},\; i=1, 2.
            \end{align*}
        Then
            \begin{align}\label{eq:alpha-1}
                \x_2^\top \distmatrix\x_1 &= \x_2^\top \iota_1 \one = \iota_1,
            \end{align}
        but $\distmatrix$ is symmetric, so $\x_2^\top \distmatrix\x_1 = \x_1^\top \distmatrix\x_2$, and thus by taking transposes of~\cref{eq:alpha-1}, it follows that $\iota_1 = \iota_2$.
    \end{proof}

    The following lemma immediately highlights why this notion is useful for investigating distance exceptional graphs.

    \begin{lemma}\label{lem:beta-zero-is-dx}
        Let $G$ be a connected graph. Then $G$ is distance exceptional if and only if $\iota(G) = 0$.
    \end{lemma}

    The proof of~\cref{lem:beta-zero-is-dx} amounts to a trivial restatement of~\cref{lem:dne-positive-sum-kernel} and is omitted. The next proposition asserts that $\iota(G)$ may be distilled from $\distmatrix$ in at least two additional ways, and also highlights its connection to the Steinerberger curvature, after which it is named. If $A$ is any matrix, we write $A^+$ to denote its Moore-Penrose inverse (see~\cite{ben1963contributions}).

    \begin{proposition}\label{prop:representations-of-index}
        Let $G$ be a connected graph on $n$ vertices. Then the following statements hold concerning $\iota(G)$.
            \begin{enumerate}[label=\emph{(\roman*)}]
                \item If $G$ is not distance exceptional and $\kappa := \distmatrix^+ (n\one)$ is the Steinerberger curvature of $G$, then
                    \begin{align}\label{eq:beta-reciprocal}
                        \iota(G) = \frac{n}{\one^\top \kappa} \in \mathbb{R}\cup\{\infty\}.
                    \end{align}
                \item Let $\lambda_1\leq\lambda_2\leq\dotsc\leq\lambda_n$ denote the eigenvalues of $\distmatrix$, together with an orthonormal set of corresponding eigenvectors $\vec{u_1}, \vec{u_2},\dotsc, \vec{u_n}\in\mathbb{R}^n$. Then
                    \begin{align}\label{eq:beta-reciprocal2}
                        \iota(G)^{-1} &= \sum_{i=1}^{n} \frac{(\vec{u_i}^\top \one)^2}{\lambda_i} \in  \mathbb{R}\cup\{\infty\}.
                    \end{align}
                We interpret the right-hand side as an extended real-valued expression, according to the conventions $\frac1\infty =0$, $\frac10 = \infty$, and $\frac00= 0$.
            \end{enumerate}
    \end{proposition}

    \begin{proof}
        We first prove statement \emph{(i)} as follows. Assume that $G$ is not distance exceptional and that $G$ admits a curvature potential with nonzero sum. Then there exists some $\x\in\mathbb{R}^{n}$ such that $\one^\top\x = 1$ and $\distmatrix\x = \iota(G)\one$. Therefore we have that
            \begin{align*}
                \one^\top\kappa &= n\one^\top (\distmatrix^+ \one) \\
                &= \frac{n}{\iota(G)}\one^\top (\distmatrix^+\distmatrix \x)
            \end{align*}
        But since $G$ is not distance exceptional, $\one\in\mathrm{range}(\distmatrix)$, so $\distmatrix^+\distmatrix\x =\x$, and thus
            \begin{align*}
                \frac{n}{\iota(G)}\one^\top (\distmatrix^+\distmatrix \x) &= \frac{n}{\iota(G)}\one^\top \x = \frac{n}{\iota(G)},
            \end{align*}
        so that~\cref{eq:beta-reciprocal} holds. Next, assume that $G$ is not distance exceptional and that all curvature potentials of $G$ have zero mean. Since $G$ is not distance exceptional, $\kappa = \distmatrix^+(n\one)$ is the unique vector of minimal Euclidean length such that $\distmatrix\kappa = n\one$. In particular, $\one\in\mathrm{range}(\distmatrix)$. But since $\one\notin D[\mathbb{R}^{n}\setminus\one^\perp]$, it must hold that $\one^\top\kappa = 0$. Thus~\cref{eq:beta-reciprocal} holds as an extended real number. Next we consider statement \emph{(ii)}. First assume that $G$ is not distance exceptional and that $G$ admits a curvature potential with nonzero sum. Then there exists $\x\in\mathbb{R}^{n}$ such that $\one^\top\x = 1$ and $\distmatrix\x = \iota(G)\one$. Since $\one\in\mathrm{range}(\distmatrix)$, $\one\perp\ker\distmatrix$, and thus it holds $\one^\top \x = \one^\top \distmatrix^+(\iota\one)$. Then we have
            \begin{align*}
                1 &= \one^\top\x = \iota(G)\one^\top (\distmatrix^+\one) = \iota(G) \sum_{i=1}^{n} \frac{(\vec{u_i}^\top \one)^2}{\lambda_i},
            \end{align*}
        where the right-hand side is finite (via the conventions of statement \emph{(ii)}) since $\mathrm{ker}(\distmatrix)\perp\mathbb{R}\one$. Therefore~\cref{eq:beta-reciprocal2} holds. In the case that $G$ is not distance exceptional and all curvature potentials of $G$ have zero mean, then there exists $\x\neq 0 $ such that $\distmatrix\x = \one$, but for any such $\x$, in particular $\distmatrix^+\one$, it holds $\one^\top\x = 0$. In turn,
            \begin{align*}
                \sum_{i=1}^{n} \frac{(\vec{u_i}^\top \one)^2}{\lambda_i} &= \one^\top (\distmatrix^+\one) = \one^\top\x = 0,
            \end{align*}
        from which~\cref{eq:beta-reciprocal2} holds using the convention $\iota(G)^{-1} = 0$. Finally, if $G$ is distance exceptional, then there exists some $\x\in\mathbb{R}^{n}$ with $\one^\top\x = 1$ and $\distmatrix\x = \vec{0}$. In particular, $\x$ belongs to the eigenspace of $\distmatrix$ with eigenvalue $\lambda=0$, which is nontrivial, and thus for at least one such eigenvector, say $\vec{u_1}$, it holds that $\vec{u_1}^\top\x \neq 0$. By convention, the right-hand side of~\cref{eq:beta-reciprocal2} evaluates to $\infty$, and the claim follows.
    \end{proof}

    We remark that graphs for which $\iota(G) = \infty$ do exist. For a minimal example, we may take $G$ to be one of the graph joins $\Complete_3 + (3\Complete_1)$ or $\Complete_2 + (4\Complete_1)$, where $n\Complete_1$ consists of $n$ disjoint and isolated vertices. Then one may verify that $\distmatrix$ is nonsingular and the image $\distmatrix^{-1}\one$ has mean zero so that indeed $\iota(G) = \infty$.

    In the remainder of this section, we provide a set of results which show that $\iota(G)$ behaves in a relatively friendly manner under graph operations, allowing us to compute its value in a number of situations.

    \begin{theorem}[Cartesian products]\label{thm:index-product}
        Let $G, H$ be connected graphs. Then $\iota(G\,\Box\,H) = \iota(G) + \iota(H)$, with $\iota(G\,\Box\,H) = \infty$ if either $\iota(G) = \infty$ or $\iota(H) = \infty$.
    \end{theorem}
    
    \begin{proof}
        First assume that $\iota(G), \iota(H) < \infty$. Let $\x_G\in\mathbb{R}^{|V(G)|}$ and $\x_H\in\mathbb{R}^{|V(H)|}$ be such that $\one^\top\x_G = 1$, $\distmatrix(G)\x_G = \iota(G)\one$, and similarly for $H$. Define $\x_{{G\,\Box\, H}} := \x_G \otimes \x_H \in \mathbb{R}^{|V(G)||V(H)|}$ where $\otimes$ denotes the Kronecker product of matrices. Then $\one^\top\x_{{G\,\Box\, H}} = 1$, and by properties of the distance matrix under Cartesian products (see~\cite{bapat2019onCartesian}), it holds that
            \begin{align*}
                \distmatrix({G\,\Box\, H})\x_{{G\,\Box\, H}} &= (\distmatrix(G)\otimes J + J \otimes \distmatrix(H))(\x_G \otimes \x_H)\\
                &= (\distmatrix(G)\x_G) \otimes (J\x_H) + (J\x_G) \otimes (\distmatrix(H)\x_H)\\
                &= \iota(G)\one \otimes \one + \one \otimes \iota(H)\one\\
                &= (\iota(G) + \iota(H))\one.
            \end{align*}
        Next assume, without loss of generality, that $\iota(G) = \infty$. If $\iota(H)\ne 0$, then by the preceding argument, $G\,\Box\,H$ is not distance exceptional and $\iota(G\,\Box\, H) \ne 0$. Otherwise, assuming $\iota(H) = 0$, we claim that $\iota(G\,\Box\, H) \ne 0$. To this extent, let $\y$ be such that $\distmatrix({G\,\Box\, H})\y = 0$. Write
            \begin{align*}
                (\y_G)_{u} &= \sum_{v\in V(H)} \y_{(u,v)},\quad u\in V(G),\\
                (\y_H)_{v} &= \sum_{u\in V(G)} \y_{(u,v)},\quad v\in V(H).
            \end{align*}
        Then
            \begin{align*}
                (\distmatrix({G\,\Box\, H})\y)_{u, v} &= \sum_{u'\in V(G)} d_G(u, u') (\y_G)_{u'} + \sum_{v'\in V(H)} d_H(v, v') (\y_H)_{v'}\\
                &= (\distmatrix(G) \y_G)_{u} + (\distmatrix(H) \y_H)_{v}.
            \end{align*}
        Thus it must hold that $\distmatrix(G) \y_G = c\one$ and $\distmatrix(H) \y_H = -c\one$ for some $c\in\mathbb{R}$. Regardless of $c$, since $\iota(G) = \infty$, it must hold that $\one^\top \y_G = \vec{0}$. Therefore
            \begin{align*}
                \one^\top \y &= \sum_{u\in V(G)} (\y_G)_u = 0,
            \end{align*}
        i.e., by~\cref{lem:dne-positive-sum-kernel}, $\iota(G\,\Box\,H)\ne 0$, i.e., $G\,\Box\,H$ admits a curvature potential. Finally, we claim that $\iota(G\,\Box\,H)= \infty$. To see this, suppose $\distmatrix({G\,\Box\,H}) \y = \one$, and let $\y_G, \y_H$ be as before. Then
            \begin{align*}
                1 &= (\distmatrix({G\,\Box\, H})\y)_{u, v} = (\distmatrix(G) \y_G)_{u} + (\distmatrix(H) \y_H)_{v},
            \end{align*}
        so that $\distmatrix(G) \y_G \propto \one$. Since $\iota(G) = \infty$, it must hold that $\one^\top \y_G = 0$, and thus
            \begin{align*}
                \one^\top \y &= \sum_{u\in V(G)} (\y_G)_u = 0.
            \end{align*}
        This completes the proof.
    \end{proof}

    We remark that~\cref{thm:index-product} recovers and generalizes the result of~\cite[Thm. 6]{dudarov2023image} which considered only the case $\iota(G) = \iota(H) = 0$. The next result concerns the procedure of connecting two graphs $G, H$ at a vertex, which was investigated in~\cite{chen2023steinerberger}.

    \begin{theorem}[Vertex coalescence]\label{thm:merging}
        Let $G, H$ be connected graphs. Let $u_{0}\in V(G)$ and $v_{0}\in V(H)$ be fixed vertices, and let $G\,\Circ_{u_{0},v_{0}}\,H$ denote the vertex coalescence of $G$ and $H$ at $u_{0}$ and $v_{0}$, respectively. Then
            \begin{align*}
                \iota(G\,\Circ_{u_{0},v_{0}}\,H) &= \iota(G) + \iota(H).
            \end{align*}
        If either $\iota(G) = \infty$ or $\iota(H) = \infty$, then $\iota(G\,\Circ_{u_{0},v_{0}}\,H)=\infty$.
    \end{theorem}

    \begin{proof}
        Fix an enumeration of $V(G)$ and $V(H)$ such that $u_0, v_0$ appear first. Write $G\,\Circ\, H := G\,\Circ_{u_{0},v_{0}}\,H$ for brevity. Assume that $\iota(G),\iota(H)<\infty$. Then there exist vectors $\x_G,\x_H$ with $\one^\top\x_G = \one^\top\x_H = 1$ and $\distmatrix(G)\x_G = \iota(G)\one$, $\distmatrix(H)\x_H = \iota(H)\one$. Write
            \begin{align*}
                \x_G^\top &= \begin{bmatrix}
                    \alpha_G & \y_G^\top
                \end{bmatrix},\quad \x_H^\top = \begin{bmatrix}
                    \alpha_H & \y_H^\top
                \end{bmatrix}, \quad \alpha_G,\alpha_H\in\mathbb{R}.
            \end{align*}
        Order the vertices of $G\,\Circ\,H$ as
            \begin{align*}
                V(G\,\Circ\,H) = \{w\} \cup (V(G)\setminus \{u_{0}\} )\cup (V(H)\setminus \{v_{0}\}),
            \end{align*}
        where $w\in V(G\,\Circ\,H)$ is the vertex appearing after $u_0, v_0$ are identified. Define $s = |V(G)| - 1$, $t = |V(H)| - 1$. Denote by ${\distmatrix(G)}'$ (resp. ${\distmatrix(H)}'$) the principal submatrix of ${\distmatrix(G)}$ (resp. ${\distmatrix(H)}$) indexed by $V(G)\setminus \{u_{0}\}$ (resp. $V(H)\setminus \{v_{0}\}$), obtained by deleting the row and column corresponding to $u_{0}$ (resp. $v_{0}$). Let $\x_{1}\in\mathbb{R}^{s}$ denote the vector of distances between $u_{0}$ and $V(G)\setminus \{u_{0}\}$, and let $\y_{1}\in\mathbb{R}^{t}$ denote the vector of distances between $v_{0}$ and $V(H)\setminus \{v_{0}\}$. With this setup, the distance matrix $\distmatrix = \distmatrix(G\,\Circ\,H)$ takes the form
            \begin{align}\label{eq:d-decomp}
                \distmatrix = \begin{bmatrix}
                    0 & \x_{1}^\top & \y_{1}^\top\\
                    \x_{1} & {\distmatrix(G)}' & \x_{1}\one^\top + \one \y_{1}^\top\\
                    \y_{1} & \one \x_{1}^\top + \y_{1}\one^\top & {\distmatrix(H)}'
                \end{bmatrix}.
            \end{align}
        since for $u\in V(G)\setminus\{u_{0}\}$ and $v\in V(H)\setminus\{v_{0}\}$ one has
        $d_{G\,\Circ\,H}(u, v)=d_G(u,u_{0})+d_H(v_{0},v)$. Now let $\x\in\mathbb{R}^{1+s+t}$ be the vector given by
            \begin{align}\label{eq:xw}
                \x &= \begin{bmatrix}
                    x_w\\ \y_G \\ \y_H
                \end{bmatrix}, \quad x_w\in\mathbb{R},\, y_G\in\mathbb{R}^{s},\,y_H\in\mathbb{R}^{t}.
            \end{align}
        To guarantee that $\one^\top x = 1$, we put
            \begin{align*}
                x_w = 1-\one^\top \y_G - \one^\top \y_H = 1 - (1-\alpha_G) - (1-\alpha_H) = \alpha_G + \alpha_H - 1.
            \end{align*}
        Then we have that
            \begin{align*}
                \distmatrix \x &= \begin{bmatrix}
                    0 & \x_{1}^\top & \y_{1}^\top\\
                    \x_{1} & {\distmatrix(G)}' & \x_{1}\one^\top + \one \y_{1}^\top\\
                    \y_{1} & \one \x_{1}^\top + \y_{1}\one^\top & {\distmatrix(H)}'
                \end{bmatrix}\begin{bmatrix}
                    x_w\\ \y_G \\ \y_H
                \end{bmatrix}\\
                &= \begin{bmatrix}
                    \x_{1}^\top \y_G + \y_{1}^\top \y_H\\
                    \x_{1} x_w + {\distmatrix(G)}' \y_G + \x_1\one^\top\y_H + \one \y_{1}^\top \y_H\\
                    \y_{1} x_w + \one \x_{1}^\top \y_G + \y_1\one^\top \y_G + {\distmatrix(H)}' \y_H
                \end{bmatrix}\\
                &= \begin{bmatrix}
                    \iota(G) + \iota(H)\\
                    \x_1\,x_w + \bigl(\iota(G)\one - \alpha_G\x_1\bigr) + \x_1(\one^\top\y_H) + \one(\y_1^\top\y_H)\\
                    \y_1\,x_w + \one(\x_1^\top\y_G) + \y_1(\one^\top\y_G) + \bigl(\iota(H)\one - \alpha_H\y_1\bigr)\\
                \end{bmatrix}.
            \end{align*}
        since, by construction, $\alpha_G \x_1 + {\distmatrix(G)}' \y_G = \iota(G)\one$, and similarly $\alpha_H \y_1 + {\distmatrix(H)}' \y_H = \iota(H)\one$. Continuing, we have
            \begin{align*}
                \begin{bmatrix}
                    \iota(G) + \iota(H)\\
                    \x_1\,x_w + \bigl(\iota(G)\one - \alpha_G\x_1\bigr) + \x_1(\one^\top\y_H) + \one(\y_1^\top\y_H)\\
                    \y_1\,x_w + \one(\x_1^\top\y_G) + \y_1(\one^\top\y_G) + \bigl(\iota(H)\one - \alpha_H\y_1\bigr)\\
                \end{bmatrix} &= \begin{bmatrix}
                    \iota(G) + \iota(H)\\
                    \bigl(\iota(G)+\iota(H)\bigr)\one \;+\; \x_1\,(x_w + \one^\top\y_H - \alpha_G)\\
                    \bigl(\iota(G)+\iota(H)\bigr)\one \;+\; \y_1\,(x_w + \one^\top\y_G - \alpha_H)\\
                \end{bmatrix}\\
                &= (\iota(G) + \iota(H))\one,
            \end{align*}
        as claimed, since
            \begin{align*}
                x_w + \one^\top\y_H - \alpha_G &= (\alpha_G + \alpha_H - 1) + (1 - \alpha_H) - \alpha_G = 0,
            \end{align*}
        and similarly for the other term. This proves the first part of the theorem.

        Next, we assume without loss of generality that $\iota(G) = \infty$. In this case, we claim that $\iota(G\,\Circ\,H) = \infty$. By assumption there exists some $\y_G^\top = [\alpha\; \y_G'^\top]$ such that $\distmatrix(G)\y_G = \one$ and $\one^\top \y_G = \alpha + \one^\top \y_G' = 0$.  Define
            \begin{align*}
                \x &= \begin{bmatrix}
                    \alpha \\ \y_G' \\ \vec{0}
                \end{bmatrix} \in\mathbb{R}^{V(G\,\Circ\,H)}.
            \end{align*}
        Then, re-using the block decomposition~\cref{eq:d-decomp}, we have
            \begin{align*}
                \distmatrix \x &= \begin{bmatrix}
                    0 & \x_{1}^\top & \y_{1}^\top\\
                    \x_{1} & {\distmatrix(G)}' & \x_{1}\one^\top + \one \y_{1}^\top\\
                    \y_{1} & \one \x_{1}^\top + \y_{1}\one^\top & {\distmatrix(H)}'
                \end{bmatrix}\begin{bmatrix}
                    \alpha \\ \y_G' \\ \vec{0}
                \end{bmatrix} &= \begin{bmatrix}
                    \x_{1}^\top \y_G'\\
                    \x_{1} \alpha + {\distmatrix(G)}' \y_G'\\
                    \y_{1} \alpha + \one \x_{1}^\top \y_G' + \y_1\one^\top\y_G'
                \end{bmatrix} = \begin{bmatrix}
                    1\\
                    \one\\
                    \one
                \end{bmatrix},
            \end{align*}
        since $\x_1 \alpha + {\distmatrix(G)}' \y_G' = \one$ by construction, and since
            \begin{align*}
                \y_{1} \alpha + \one \x_{1}^\top \y_G' + \y_1\one^\top\y_G' &= \y_{1} \alpha + \one + \y_1(-\alpha) = \one.
            \end{align*}
        But then $\one^\top \x = \alpha + \one^\top \y_G' + 0 = 0$. Thus we have shown that there exists some $\x$ with $\distmatrix(G\,\Circ\,H)\x = \one$ and $\one^\top\x = 0$. Let $\y\in\mathbb{R}^{V(G\,\Circ\,H)}$ be arbitrary such that $\distmatrix(G\,\Circ\,H)\y = \one$. Then
            \begin{align*}
                \distmatrix(G\,\Circ\,H)(\y - \x) = \vec{0},
            \end{align*}
        and thus $\y-\x \in \mathrm{ker}(\distmatrix(G\,\Circ\,H))$. Since $\distmatrix(G\,\Circ\,H)$ is symmetric and $\one\in\mathrm{range}(\distmatrix(G\,\Circ\,H))$, it holds that $\mathrm{ker}(\distmatrix(G\,\Circ\,H)) \perp \one$. Therefore
            \begin{align*}
                \one^\top \y &= \one^\top \x + \one^\top (\y - \x) = 0 + 0 = 0,
            \end{align*}
        from which it follows that $\iota(G\,\Circ\,H) = \infty$.
    \end{proof}

    We remark that~\cref{thm:merging} recovers and generalizes that of~\cite{chen2023steinerberger}, which considers the case $\iota(G) = \iota(H) = 0$. Finally, to conclude this section, we consider the case of graph joins. To do so properly, we introduce a bit of notation. Letting $G$ be a (possibly disconnected) graph on $n$ vertices, begin by writing $\widetilde{D}(G)\in\mathbb{R}^{n\times n}$ to denote the principal submatrix of the distance matrix of the cone $G+\Complete_1$ of $G$, i.e., if $G$ is connected, we set
            \begin{align*}
                \widetilde{D}(G) &= \min\{\distmatrix(G), 2J\}
            \end{align*}
    where $J$ is the all-ones matrix. Let $\widetilde{\iota}(G)\in\mathbb{R}$ be defined as the unique scalar for which $\widetilde{D}(G)\x = \widetilde{\iota}(G)\one$ among those $\x$ with $\one^\top \x = 1$, with the same conventions at $\infty$ for the curvature index $\iota(\cdot)$. We refer to $\widetilde{\iota}(G)$ as the \emph{modified curvature index}, and its uniqueness follows from the symmetry of $\widetilde{D}(G)$ exactly as in the proof of~\cref{prop:index-of-curvature}.

    \begin{remark}\label{rmk:diam2-modified-index}
        If $\mathrm{diam}(G)\le 2$, then $\widetilde{\iota}(G) = \iota(G)$.
    \end{remark}

    \begin{theorem}[Graph joins]\label{thm:index-join}
        Let $G, H$ be graphs, connected or otherwise. Then 
            {\footnotesize\begin{align}\label{eq:beta-join}
                \iota(G+ H) &= \begin{cases}
                    \dfrac{\widetilde{\iota}(G) \widetilde{\iota}(H)-1}{\,\widetilde{\iota}(G)+\widetilde{\iota}(H)-2\,}, & \widetilde{\iota}(G),\widetilde{\iota}(H)<\infty\ \text{and}\ \widetilde{\iota}(G)+\widetilde{\iota}(H)\neq 2,\\
                    1, & \widetilde{\iota}(G)=\widetilde{\iota}(H)=1,\\
                    \infty, & \widetilde{\iota}(G)+\widetilde{\iota}(H)=2,\ (\widetilde{\iota}(G),\widetilde{\iota}(H))\neq(1, 1),\\
                    \widetilde{\iota}(H), & \widetilde{\iota}(G)=\infty,\ \widetilde{\iota}(H)<\infty,\\
                    \widetilde{\iota}(G), & \widetilde{\iota}(H)=\infty,\ \widetilde{\iota}(G)<\infty,\\
                    \infty, & \widetilde{\iota}(G)=\widetilde{\iota}(H)=\infty.
                \end{cases}
            \end{align}}
    \end{theorem}
    
    \begin{proof}
        For a fixed vector $\x \in\mathbb{R}^{|V(G)|+|V(H)|}$, partition $\x$ and label its components in the following manner:
            \begin{align*}
                \x = \begin{bmatrix}
                    \x_G\\\x_H
                \end{bmatrix}
            \end{align*}
        Write $s = \one^\top\x_G$ and $t = \one^\top \x_H$. By properties of the distance matrix under graph joins (see, e.g., ~\cite{dudarov2023image,stevanovic2009distance}), it holds that
            \begin{align*}
                \distmatrix(G+ H)\x_{G+ H} &= \begin{bmatrix}
                    \widetilde{D}(G) & J\\
                    J & \widetilde{D}(H)
                \end{bmatrix}\begin{bmatrix}
                    \x_G\\\x_H
                \end{bmatrix} = \begin{bmatrix}
                    \widetilde{D}(G)\x_G + ( \one^\top\x_H)\one\\
                    ( \one^\top\x_G)\one + \widetilde{D}(H)\x_H
                \end{bmatrix} = \begin{bmatrix}
                    \widetilde{D}(G)\x_G + t\one\\
                    s\one + \widetilde{D}(H)\x_H
                \end{bmatrix}.
            \end{align*}
        Now $\one^\top \x = 1$ if and only if $s+t = 1$, and $\distmatrix(G+ H)\x_{G+ H} = \iota\one$ if and only if 
            \begin{align*}
                \begin{bmatrix}
                    \widetilde{D}(G)\x_G + t\one\\
                    s\one + \widetilde{D}(H)\x_H
                \end{bmatrix}&= \begin{bmatrix}
                    \iota\one\\
                    \iota\one
                \end{bmatrix}.
            \end{align*}
        The first row of this equation reads $\widetilde{D}(G)\x_G + t\one = \iota\one$, so that $\widetilde{D}(G)\x_G = (\iota - t)\one$. By definition of $\widetilde{\iota}(G)$, this is solvable with $\one^\top x_G = s$ only if $\iota - t = s\widetilde{\iota}(G)$. Similarly, from the second row we have that $\widetilde{D}(H)\x_H = (\iota - s)\one$, so that $\iota - s = t\widetilde{\iota}(H)$. Putting these equations together, we have
            \begin{align}\label{eq:stbeta}
                \begin{cases}
                    s+t &=1,\\
                    \iota - t &= s\widetilde{\iota}(G),\\
                    \iota - s &= t\widetilde{\iota}(H).
                \end{cases}
            \end{align}
        If $\widetilde{\iota}(G),\widetilde{\iota}(H)<\infty$, this can be solved for $\iota$ to recover the first three cases of~\cref{eq:beta-join}. Note the special case of $\widetilde{\iota}(G) = \widetilde{\iota}(H) = 1$, which admits the solution $s=t=1/2$ and $\iota(G+ H) = 1$. If $\widetilde{\iota}(G) = \infty$ and $\widetilde{\iota}(H) < \infty$, then we have the reduced system
            \begin{align*}
                \begin{cases}
                    s+t &=1,\\
                    \iota - s &= t\widetilde{\iota}(H).
                \end{cases}
            \end{align*}
        into which we can put $s=0, t=1$ to recover the fourth case~\cref{eq:beta-join}. The fifth case is similar, and the final case follows from the fact that if both modified indices are infinite, then we force $s=t=0$ and no solution to~\cref{eq:stbeta} exists.
    \end{proof}

    \cref{thm:index-join} recovers and generalizes the result of~\cite[Thm. 4]{dudarov2023image}, which considered the case of $\widetilde{\iota}(G) = 0$ and $\widetilde{\iota}(H) = \infty$.
    
    \begin{remark}
        Although we have stated~\cref{thm:index-join} in several cases, these can all be seen as limits of the first case if one is generous. For example, we have that
            \begin{align*}
                \lim_{\widetilde{\iota}(G) \to \infty}\dfrac{\widetilde{\iota}(G) \widetilde{\iota}(H)-1}{\,\widetilde{\iota}(G)+\widetilde{\iota}(H)-2\,} = \widetilde{\iota}(H),
            \end{align*}
        provided $\widetilde{\iota}(H)<\infty$ and $\widetilde{\iota}(G)+\widetilde{\iota}(H)-2\ne 0$. The special case $\widetilde{\iota}(G)=\widetilde{\iota}(H)=1$ amounts to a removable singularity of the function $x\mapsto \frac{x^2-1}{2(x-1)}$. The other cases can be handled in similar ways.
    \end{remark}

    \section{Computing the curvature index for graph families}\label{sec:computing-index}

    In this section we compute the curvature index for various graph families, which will be instructive for modeling the usage of the graph operations results in the previous section, and will be useful for constructing embeddings in the following section. Our first result concerns trees and symmetric graphs. Before stating it we recall the \emph{Wiener index} of a graph, denoted $W(G)$, which is given by
        \begin{align*}
            W(G) &= \frac{1}{2} \sum_{u, v\in V} d(u, v).
        \end{align*}

    \begin{theorem}\label{thm:transitive-tree}
        Let $G$ be a connected graph on $n$ vertices. Then the following hold:
            \begin{enumerate}[label=(\roman*)]
                \item If $G$ is distance regular, then $\iota(G) = \frac{2W(G)}{n^2}$;
                \item If $G$ is a tree, $\iota(G) = \frac{n-1}{2}$.
            \end{enumerate}
    \end{theorem}

    \begin{proof}
        For the first claim, since $G$ is distance regular, each row sum of $\distmatrix$ is equal to $\frac{2W(G)}{n}$. Therefore, taking $\x = \frac{1}{n}\one$, we have that
            \begin{align*}
                \distmatrix\x = \distmatrix \left(\tfrac{1}{n}\one\right) = \frac{2W(G)}{n^2}\one,
            \end{align*}
        and the claim follows by definition of $\iota(G)$. For the second claim, let $G$ be a tree on $n$ vertices. It is well known that the distance matrix of a tree is invertible (see~\cite{graham1978distance}), so $\iota(G)\ne 0$. Let $\kappa = \distmatrix^{-1}(n\one)$ denote the Steinerberger curvature of $G$. By~\cite{chen2023steinerberger} (see also~\cite{robertson2024discrete}), we have that
            \begin{align*}
                \kappa_u &= \frac{n}{n-1}(2-d(u)),
            \end{align*}
        so that by~\cref{prop:representations-of-index}, it holds
            \begin{align*}
                \iota(G) &= \frac{n-1}{\sum_{u\in V(G)} (2-d(u))} = \frac{n-1}{2}, 
            \end{align*}
        as claimed. The theorem is proved.
    \end{proof}

    The Wiener indices of many families of graphs are well known (see, e.g., the survey~\cite{knor2019some}), and thus~\cref{thm:transitive-tree}, along with~\cref{thm:index-product}, allows us to identify the curvature indices for many families of graphs. We summarize these in a theorem below, which follows directly in the footsteps of the preceding results and is stated without proof.

    \begin{theorem}\label{thm:examples}
        \begin{enumerate}[label=(\roman*)]
            \item For $n\ge1$, let $\Cycle_n$ denote the $n$-cycle. Then
                \begin{align*}
                    \iota(\Cycle_n) &= \begin{cases}
                        \frac{n}{4}, & n\ \text{even},\\
                        \frac{n^2-1}{4n}, & n\ \text{odd}.
                    \end{cases}
                \end{align*}
            \item For $n\ge 2$ and $d\geq 1$, the $d$-torus $\Cycle_n^{\Box d}$ satisfies
                \begin{align*}
                    \iota(\Cycle_n^{\Box d}) &= \begin{cases}
                        \frac{dn}{4}, & n\ \text{even},\\
                        \frac{d(n^2-1)}{4n}, & n\ \text{odd}.
                    \end{cases}
                \end{align*}
            \item For $d\ge 2$, let $\Cube_d = \Complete_2^{\Box d}$ denote the graph of the $d$-cube. Then
                \begin{align*}
                    \iota(\Cube_d) &= \frac{d}{2}.
                \end{align*}    
            \item For $n, m\geq 1$, the rectangular lattice $\Path_n \,\Box\, \Path_m$ satisfies
                \begin{align*}
                    \iota(\Path_n\,\Box\, \Path_m) &= \frac{n-1}{2} + \frac{m-1}{2}.
                \end{align*}
            \item For $n\geq 1$, the complete graph $\Complete_n$ satisfies
                \begin{align*}
                    \iota(\Complete_n) &= \frac{n-1}{n}.
                \end{align*}
        \end{enumerate}
    \end{theorem}

    The case of complete multipartite graphs is particularly useful because they provide a source of negative curvature indices. Computing their curvature indices requires some care and can be carried out using the graph join and~\cref{thm:index-join}.
    
    \begin{theorem}\label{thm:kmn}
        For $k \geq 2$, let $a_1,\dotsc, a_k\ge 1$ be fixed natural numbers, and consider the complete multipartite graph $\Complete_{a_1,\dotsc, a_k}$. Then
            \begin{align}\label{eq:kmn-1}
                \iota(\Complete_{a_1,\dotsc, a_k}) &= \begin{cases}
                   1 + \left(\sum_{i=1}^{k} \frac{a_i}{a_i-2}\right)^{-1} & \text{ if no }a_i = 2,\\
                   1 & \text{ if some }a_i = 2,
                \end{cases}
            \end{align}
        Where $\iota(\Complete_{a_1,\ldots, a_k}) = \infty$ when $\sum_{i=1}^k \frac{a_i}{a_i - 2} = 0$.
        In particular, if $k\geq 5$, then
            \begin{align*}
                \iota(\Complete_{1, 1, k}) &= -\frac{2}{k-4}.
            \end{align*}
    \end{theorem}

    \begin{proof}
        We will prove the statement for all $k \geq 1$, replacing $\iota(\Complete_{a_1,\ldots,a_k})$ with $\widetilde{\iota}(\Complete_{a_1,\ldots,a_k})$ and noting that when $k \geq 2$ these are equal by~\cref{rmk:diam2-modified-index}, since $\mathrm{diam}(\Complete_{a_1,\ldots,a_k}) = 2$ for $k \geq 2$.
        We first observe that a complete multipartite graph is the graph join of independent sets, that is 
            \begin{align*}
                \Complete_{a_1,\ldots,a_k} = a_1\Complete_1 + a_2 \Complete_1 + \cdots + a_k\Complete_1,
            \end{align*}
        where $a\Complete_1$ denotes a collection of $a$ isolated vertices. We will use induction and apply~\cref{thm:index-join}. For the base case, let $a\ge 1$ be a fixed natural number. It follows that $\widetilde{D}(a\Complete_1) = 2(J-I)$, and letting $\vec{x} = \frac{1}{a}\vec{1}$ we have $\vec{1}^\top\vec{x} = 1$ and $\widetilde{D}(a\Complete_1)\vec{x} = \frac{2(a-1)}{a}\vec{1}$, so $\widetilde{\iota}(a\Complete_1) = 2 - 2/a$, and thus~\cref{eq:kmn-1} holds. Now suppose~\cref{eq:kmn-1} holds for any fixed choice of $k-1$ natural numbers $a_1,\dotsc,a_{k-1}$, and let $a_{k}$ be fixed. We first check a few edge cases and then verify the general formula at the end. First, in the case that any $a_i$ is $2$, we have either $\iota(\Complete_{a_1,\ldots,a_{k-1}}) = 1$ by induction or $\widetilde{\iota}(a_k\Complete_1) = 1$ by the preceding calculation, and we get $\iota(\Complete_{a_1,\ldots,a_k}) = 1$ as desired. Next, assume no $a_i=2$. If
            \begin{align}\label{eq:kmn-2}
                \iota(\Complete_{a_1,\ldots,a_{k-1}}) + \widetilde{\iota}(a_k\Complete_1) = 2,
            \end{align}
        we must also have that $\widetilde{\iota}(a_k\Complete_1) \neq 1$ since $a_k\neq 2$, so by~\cref{thm:index-join}, $\iota(\Complete_{a_1,\ldots,a_k}) = \infty$.~\cref{eq:kmn-2} is equivalent to $\sum_{i=1}^k\frac{a_i}{a_i - 2} = 0$, so~\cref{eq:kmn-1} holds with the convention $\tfrac{1}{0}=\infty$. For the final edge case, if $\iota(\Complete_{a_1,\ldots,a_{k-1}}) = \infty$ we must have $\sum_{i=1}^{k-1}\frac{a_k}{a_k - 2} = 0$, so that by~\cref{thm:index-join},
            \begin{align*}
                \iota(\Complete_{a_1,\ldots,a_k} + a_k\Complete_1) = \widetilde{\iota}(a_k\Complete_1) = 2 - \frac{2}{a_k} = 1 + \left(\left[\sum_{i=1}^{k-1}\frac{a_k}{a_k - 2}\right]+\frac{a_k}{a_k-2}\right)^{-1},
            \end{align*}
        and~\cref{eq:kmn-1} holds. Lastly, we check the general case, assuming $\widetilde{\iota}(\Complete_{a_1,\dotsc, a_k}), \widetilde{\iota}(a_k\Complete_1) < \infty$ and that~\cref{eq:kmn-2} does not hold. Then~\cref{thm:index-join} gives
            \begin{align*}
                \iota(\Complete_{a_1,\ldots,a_k}) &= \frac{\left(1 + \left(\sum_{i=1}^{k-1} \frac{a_i}{a_i - 2}\right)^{-1}\right)(2 - 2/a) -1}{1 + \left(\sum_{i=1}^{k-1} \frac{a_i}{a_i - 2}\right)^{-1} +2 - 2/a - 2} \\
                &=1 + \frac{a_k - 2}{(a_k - 2){\sum_{i=1}^{k-1} \frac{a_i}{a_i - 2}} + a_k} = 1 + \left(\sum_{i=1}^k \frac{a_i}{a_i - 2}\right)^{-1}.
            \end{align*}
        By induction, the theorem is proved.
    \end{proof}

    \section{Constructions of embeddings into distance exceptional graphs}\label{sec:embeddings}

    In this section, we apply the results of the previous sections to construct embeddings of various graphs into distance exceptional graphs. We begin with a useful observation that the obstruction $\iota(G) = \infty$ can be easily resolved at the expense of a nonisometric embedding.

    \begin{lemma}\label{lem:infinite-obstruction}
        Let $G=(V,E)$ be any graph on $n$ vertices (connected or otherwise). Then there exists a connected graph $G'$ on at most $n+2$ vertices containing an induced copy of $G$ satisfying $\iota(G') < \infty$. 
    \end{lemma}

    \begin{proof}
        If $\iota(G) < \infty$ take $G'=G$. Otherwise if $\iota(G) = \infty$, let $G' = G + 2\Complete_1$ be the join of $G$ with two isolated vertices. As in the proof of~\cref{thm:kmn}, we have $\widetilde{\iota}(2\Complete_1) = 1$, so by~\cref{thm:index-join}, $\iota(G + 2\Complete_1) = 1 < \infty$.
    \end{proof}

    The next result states that $\iota(\cdot)$ exhausts all rational numbers. In turn, for any $G$ with $\iota(G) < \infty$, there exists a graph $G'$ with $\iota(G') = -\iota(G)$ so that both $G\,\Box\,G'$ and $G\Circ_{u, v} G'$ (for any choice of $u\in V(G), v\in V(G')$) are distance exceptional.

    \begin{theorem}\label{thm:rational-index}
        Let $q\in\mathbb{Q}$. Then there exists a graph $G$ such that $\iota(G) = q$.
    \end{theorem}

    \begin{proof}
        If $q=0$, supply any distance exceptional graph and the claim follows. Suppose $q < 0$. Write $-q$ as an Egyptian fraction:
            \begin{align*}
                -q = \sum_{i=1}^{k} \frac{1}{n_i},\quad n_i\in\mathbb{N},\,1\leq i \leq k.
            \end{align*}
        Then we have that, by~\cref{thm:kmn}, it holds
            \begin{align*}
                \iota(\Complete_{1, 1, 2n_i+4}) = -\frac{1}{n_i}.
            \end{align*}
        Therefore
            \begin{align*}
                \iota\left(\Box_{i=1}^{k} \Complete_{1, 1, 2n_i+4} \right) = q.
            \end{align*}
        Alternatively, one may take $\Complete_{1, 1, 2n_1+4} \Circ_{u_1, v_1} \Complete_{1, 1, 2n_2+4} \cdots \Circ_{u_{k-1}, v_{k-1}} \Complete_{1, 1, 2n_k+4}$ in any manner to arrive at the same result. Now suppose $q>0$. Find $m\in\mathbb{N}$ such that $m>q$. Write $m-q =: r \in \mathbb{Q}_{>0}$. Then find an Egyptian fraction for $r$:
            \begin{align*}
                r = \sum_{i=1}^{k} \frac{1}{n_i},\quad n_i\in\mathbb{N},\,1\leq i \leq k.
            \end{align*}
        Put
            \begin{align*}
                G = \left(\Box_{i=1}^{k} \Complete_{1, 1, 2n_i+4}\right) \;\;\square\;\; \Path_3^{\Box m}.
            \end{align*}
        Since $\iota(\Path_3) = 1$ via~\cref{thm:transitive-tree} and $\iota(\Complete_{1, 1, 2n_i+4}) = -\frac{1}{n_i}$ via~\cref{thm:kmn} we have that
            \begin{align*}
                \iota(G) &= m - \sum_{i=1}^{k} \frac{1}{n_i} = m-r = q.
            \end{align*}
        The theorem is proved.
    \end{proof}

    Below we state and prove our main result for this section.

    \begin{theorem}\label{thm:embedding}
        Let $G$ be a graph, connected or otherwise. Then there exists a distance exceptional graph $G'$ containing a copy of $G$ as an induced subgraph. If $G$ is connected and $\iota(G) < \infty$, then this embedding can be taken to be isometric. 
    \end{theorem}

    \begin{proof}
        By~\cref{lem:infinite-obstruction}, there exists a graph $G'$ containing an induced copy of $G$ which satisfies $\iota(G')<\infty$. By the proof of~\cref{lem:infinite-obstruction}, if $\iota(G)<\infty$, we can take $G'=G$. By~\cref{thm:rational-index}, there exists a graph $H$ such that $\iota(H) = -\iota(G')$. Fix any vertex $u$ of $G'$ and any vertex $v$ of $H$, and consider $G'' = G' \Circ_{u, v} H$. Then $G''$ contains an induced copy of $G'$, and this embedding of $G'$ into $G''$ is an isometry of $G'$. By~\cref{thm:merging}, $\iota(G'') = \iota(G) - \iota(G) = 0$, so by~\cref{lem:beta-zero-is-dx}, $G''$ is distance exceptional. The theorem is proved.
    \end{proof}

    The proofs of~\cref{lem:infinite-obstruction},~\cref{thm:rational-index}, and~\cref{thm:embedding} are constructive, so we can state an algorithm for constructing distance exceptional graphs containing embedded copies of given ones.
    
    \begin{algorithm}[H]
        \caption{ Iterative procedure for creating a distance exceptional graph with given induced subgraph }\label{alg:X}
        \begin{algorithmic}[1]
        \REQUIRE $G$ is a given graph
        \IF{$G$ is disconnected or $\iota(G) = \infty$}
        \STATE $G'\gets G+ (\Complete_1\cup \Complete_1)$
        \ELSE 
        \STATE $G'\gets G$
        \ENDIF
        \WHILE{$\iota(G') < 0$}
        \STATE $G' \gets\texttt{merge}(G', \Complete_2)$
        \ENDWHILE
        \STATE $(n_1,\dotsc, n_k) \gets \mathtt{egyptian\_fraction}(\iota(G'))$
        \FOR{$i=1,\dotsc, k$}
        \STATE $G'\gets\mathtt{merge}(G', \Complete_{1, 1, 2 n_i+4})$
        \ENDFOR
        \RETURN $G'$
        \end{algorithmic}
    \end{algorithm}

    In~\cref{alg:X}, $\mathtt{egyptian\_fraction}(r)$ is any function which returns a $k$-tuple of natural numbers $n_1,\dotsc, n_k$ such that
        \begin{align*}
            r = \sum_{i=1}^{k}\frac{1}{n_i}, \quad r\in\mathbb{Q}_{>0}.
        \end{align*}
    The function $\mathtt{merge}(G, H)$ selects two distinguished vertices $u\in V(G)$ and $v\in V(H)$ in an arbitrary manner and returns the graph $G\,\Circ_{u, v} H$. In the remainder of this section, we highlight a few examples of~\cref{alg:X} in action.

    \begin{figure}
        \begin{center}
            \input{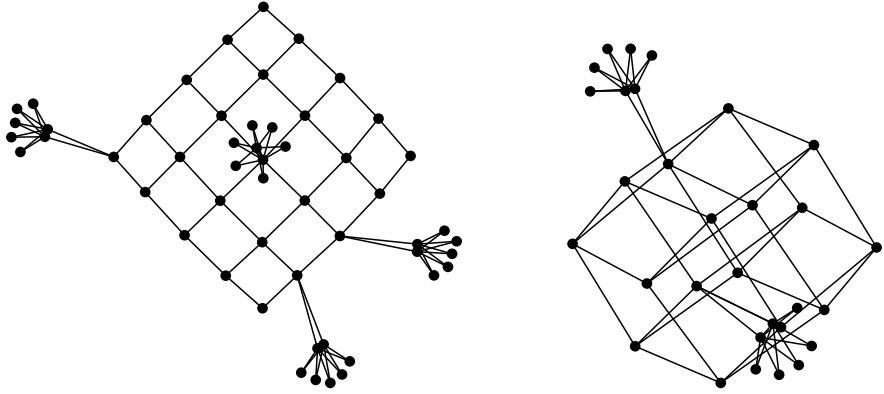}
            \hspace{1cm}
            \begin{tikzpicture}
	\begin{scope}[shift={(0, 0.0)}, scale=2.0]
		\coordinate (V-0000-0) at (-0.511, -0.251) circle (2pt);
		\coordinate (V-0000-1) at (0.184, 0.268) circle (2pt);
		\coordinate (V-0000-2) at (-1.0, 0.012) circle (2pt);
		\coordinate (V-0000-3) at (-0.371, 0.541) circle (2pt);
		\coordinate (V-0000-4) at (-0.086, 0.18) circle (2pt);
		\coordinate (V-0000-5) at (0.588, 0.667) circle (2pt);
		\coordinate (V-0000-6) at (-0.656, 0.426) circle (2pt);
		\coordinate (V-0000-7) at (0.025, 0.91) circle (2pt);
		\coordinate (V-0000-8) at (-0.025, -0.91) circle (2pt);
		\coordinate (V-0000-9) at (0.656, -0.426) circle (2pt);
		\coordinate (V-0000-10) at (-0.588, -0.667) circle (2pt);
		\coordinate (V-0000-11) at (0.086, -0.18) circle (2pt);
		\coordinate (V-0000-12) at (0.371, -0.541) circle (2pt);
		\coordinate (V-0000-13) at (1.0, -0.012) circle (2pt);
		\coordinate (V-0000-14) at (-0.184, -0.268) circle (2pt);
		\coordinate (V-0000-15) at (0.511, 0.251) circle (2pt);
		\coordinate (V-0000-16) at (-0.589, 1.04) circle (2pt);
		\coordinate (V-0000-17) at (-0.652, 1.026) circle (2pt);
		\coordinate (V-0000-18) at (-0.478, 1.261) circle (2pt);
		\coordinate (V-0000-19) at (-0.856, 1.18) circle (2pt);
		\coordinate (V-0000-20) at (-0.618, 1.305) circle (2pt);
		\coordinate (V-0000-21) at (-0.884, 1.023) circle (2pt);
		\coordinate (V-0000-22) at (-0.77, 1.304) circle (2pt);
		\coordinate (V-0000-23) at (0.237, -0.609) circle (2pt);
		\coordinate (V-0000-24) at (0.318, -0.517) circle (2pt);
		\coordinate (V-0000-25) at (0.488, -0.792) circle (2pt);
		\coordinate (V-0000-26) at (0.205, -0.82) circle (2pt);
		\coordinate (V-0000-27) at (0.478, -0.413) circle (2pt);
		\coordinate (V-0000-28) at (0.573, -0.665) circle (2pt);
		\coordinate (V-0000-29) at (0.359, -0.856) circle (2pt);
		\fill[black] (V-0000-0) circle (1pt);
		\fill[black] (V-0000-1) circle (1pt);
		\fill[black] (V-0000-2) circle (1pt);
		\fill[black] (V-0000-3) circle (1pt);
		\fill[black] (V-0000-4) circle (1pt);
		\fill[black] (V-0000-5) circle (1pt);
		\fill[black] (V-0000-6) circle (1pt);
		\fill[black] (V-0000-7) circle (1pt);
		\fill[black] (V-0000-8) circle (1pt);
		\fill[black] (V-0000-9) circle (1pt);
		\fill[black] (V-0000-10) circle (1pt);
		\fill[black] (V-0000-11) circle (1pt);
		\fill[black] (V-0000-12) circle (1pt);
		\fill[black] (V-0000-13) circle (1pt);
		\fill[black] (V-0000-14) circle (1pt);
		\fill[black] (V-0000-15) circle (1pt);
		\fill[black] (V-0000-16) circle (1pt);
		\fill[black] (V-0000-17) circle (1pt);
		\fill[black] (V-0000-18) circle (1pt);
		\fill[black] (V-0000-19) circle (1pt);
		\fill[black] (V-0000-20) circle (1pt);
		\fill[black] (V-0000-21) circle (1pt);
		\fill[black] (V-0000-22) circle (1pt);
		\fill[black] (V-0000-23) circle (1pt);
		\fill[black] (V-0000-24) circle (1pt);
		\fill[black] (V-0000-25) circle (1pt);
		\fill[black] (V-0000-26) circle (1pt);
		\fill[black] (V-0000-27) circle (1pt);
		\fill[black] (V-0000-28) circle (1pt);
		\fill[black] (V-0000-29) circle (1pt);
		\draw[line width=0.2mm, black] (V-0000-0) -- (V-0000-8);
		\draw[line width=0.2mm, black] (V-0000-0) -- (V-0000-4);
		\draw[line width=0.2mm, black] (V-0000-0) -- (V-0000-2);
		\draw[line width=0.2mm, black] (V-0000-0) -- (V-0000-1);
		\draw[line width=0.2mm, black] (V-0000-1) -- (V-0000-9);
		\draw[line width=0.2mm, black] (V-0000-1) -- (V-0000-5);
		\draw[line width=0.2mm, black] (V-0000-1) -- (V-0000-3);
		\draw[line width=0.2mm, black] (V-0000-2) -- (V-0000-10);
		\draw[line width=0.2mm, black] (V-0000-2) -- (V-0000-6);
		\draw[line width=0.2mm, black] (V-0000-2) -- (V-0000-3);
		\draw[line width=0.2mm, black] (V-0000-3) -- (V-0000-11);
		\draw[line width=0.2mm, black] (V-0000-3) -- (V-0000-7);
		\draw[line width=0.2mm, black] (V-0000-3) -- (V-0000-16);
		\draw[line width=0.2mm, black] (V-0000-3) -- (V-0000-17);
		\draw[line width=0.2mm, black] (V-0000-4) -- (V-0000-12);
		\draw[line width=0.2mm, black] (V-0000-4) -- (V-0000-6);
		\draw[line width=0.2mm, black] (V-0000-4) -- (V-0000-5);
		\draw[line width=0.2mm, black] (V-0000-5) -- (V-0000-13);
		\draw[line width=0.2mm, black] (V-0000-5) -- (V-0000-7);
		\draw[line width=0.2mm, black] (V-0000-6) -- (V-0000-14);
		\draw[line width=0.2mm, black] (V-0000-6) -- (V-0000-7);
		\draw[line width=0.2mm, black] (V-0000-7) -- (V-0000-15);
		\draw[line width=0.2mm, black] (V-0000-8) -- (V-0000-12);
		\draw[line width=0.2mm, black] (V-0000-8) -- (V-0000-10);
		\draw[line width=0.2mm, black] (V-0000-8) -- (V-0000-9);
		\draw[line width=0.2mm, black] (V-0000-9) -- (V-0000-13);
		\draw[line width=0.2mm, black] (V-0000-9) -- (V-0000-11);
		\draw[line width=0.2mm, black] (V-0000-10) -- (V-0000-14);
		\draw[line width=0.2mm, black] (V-0000-10) -- (V-0000-11);
		\draw[line width=0.2mm, black] (V-0000-11) -- (V-0000-15);
		\draw[line width=0.2mm, black] (V-0000-12) -- (V-0000-14);
		\draw[line width=0.2mm, black] (V-0000-12) -- (V-0000-13);
		\draw[line width=0.2mm, black] (V-0000-13) -- (V-0000-15);
		\draw[line width=0.2mm, black] (V-0000-14) -- (V-0000-15);
		\draw[line width=0.2mm, black] (V-0000-14) -- (V-0000-23);
		\draw[line width=0.2mm, black] (V-0000-14) -- (V-0000-24);
		\draw[line width=0.2mm, black] (V-0000-16) -- (V-0000-17);
		\draw[line width=0.2mm, black] (V-0000-16) -- (V-0000-18);
		\draw[line width=0.2mm, black] (V-0000-16) -- (V-0000-19);
		\draw[line width=0.2mm, black] (V-0000-16) -- (V-0000-20);
		\draw[line width=0.2mm, black] (V-0000-16) -- (V-0000-21);
		\draw[line width=0.2mm, black] (V-0000-16) -- (V-0000-22);
		\draw[line width=0.2mm, black] (V-0000-17) -- (V-0000-18);
		\draw[line width=0.2mm, black] (V-0000-17) -- (V-0000-19);
		\draw[line width=0.2mm, black] (V-0000-17) -- (V-0000-20);
		\draw[line width=0.2mm, black] (V-0000-17) -- (V-0000-21);
		\draw[line width=0.2mm, black] (V-0000-17) -- (V-0000-22);
		\draw[line width=0.2mm, black] (V-0000-23) -- (V-0000-24);
		\draw[line width=0.2mm, black] (V-0000-23) -- (V-0000-25);
		\draw[line width=0.2mm, black] (V-0000-23) -- (V-0000-26);
		\draw[line width=0.2mm, black] (V-0000-23) -- (V-0000-27);
		\draw[line width=0.2mm, black] (V-0000-23) -- (V-0000-28);
		\draw[line width=0.2mm, black] (V-0000-23) -- (V-0000-29);
		\draw[line width=0.2mm, black] (V-0000-24) -- (V-0000-25);
		\draw[line width=0.2mm, black] (V-0000-24) -- (V-0000-26);
		\draw[line width=0.2mm, black] (V-0000-24) -- (V-0000-27);
		\draw[line width=0.2mm, black] (V-0000-24) -- (V-0000-28);
		\draw[line width=0.2mm, black] (V-0000-24) -- (V-0000-29);
	\end{scope}
\end{tikzpicture}
        \end{center}
        \caption{This figure contains two distance exceptional graphs. \emph{(Left)} A distance exceptional graph containing an induced copy of the square lattice $\Path_5\, \Box\,\Path_5$, formed by merging $4$ copies of $\Complete_{1, 1, 6}$ to various vertices in an arbitrary manner. \emph{(Right)} A distance exceptional graph containing an induced copy of $\Cube_4$, formed by merging $2$ copies of $\Complete_{1, 1, 6}$ to various vertices in an arbitrary manner.}\label{fig:grid-jailbreak}
    \end{figure}

    \begin{example}[Square lattice graphs]\label{ex:lattice}\normalfont
        Consider the rectangular lattice graph $G=\Path_{n}\,\Box\,\Path_{n}$. By~\cref{thm:examples}, we have that $\iota(G) = n-1$. Applying~\cref{alg:X}, since $\iota(G) > 0$, we skip to line 7. An Egyptian fraction for $n-1$ is given by
            \begin{align*}
                n-1 &= \sum_{i=1}^{n-1} (1),\quad n_i = 1,\quad 1\leq i\leq n-1.
            \end{align*}
        Therefore, we perform $n-1$ merges with $\Complete_{1, 1, 6}$ to obtain a distance exceptional graph containing an isometric copy of $\Path_n\,\Box\,\Path_n$. We illustrate an example of a distance exceptional graph obtained via this construction, with $n=5$, in~\cref{fig:grid-jailbreak}.
    \end{example}

    \begin{example}[Hypercube graphs]\normalfont
        Let $k\ge 1$ and consider the hypercube graph $\Cube_{2k}$. Then by~\cref{thm:examples}, $\Cube_{2k}$ satisfies $\iota(\Cube_{2k}) = k$. Applying the same setup as in~\cref{ex:lattice}, it follows that we may construct a distance exceptional graph containing an isometric copy of $\Cube_{2k}$ by merging $k$ copies of $\Complete_{1, 1, 6}$. An example of such a distance exceptional graph containing an isometric copy of $\Cube_4$ is illustrated in~\cref{fig:grid-jailbreak}.
    \end{example}

    \section*{Acknowledgements}

    The authors gratefully acknowledge support from the Mathematics Graduate Student Council at UC San Diego in the form of their fourth-annual \emph{Mathathon} workshop at which this article was initiated. The authors also gratefully acknowledge Stefan Steinerberger for helpful comments on the manuscript during its preparation. The authors have no conflicts of interest, financial or otherwise, to disclose.

    \appendix

    \section{Computing the curvature index of a basket graph}\label{sec:appendix-basket}

    In this section, we prove~\cref{lem:curvature-basket}. By convention, we enumerate the vertices of $\Basket_k$ beginning with the first shared endpoint, followed by the $k-2$ internal vertices that remain from the copy $\Path_{k}$, and then the second shared endpoint. These are then followed by the internal vertices of each copy of $\Path_{k+1}$ (see~\cref{fig:basket}). 

    \begin{proof}[Proof of~\cref{lem:curvature-basket}]
        Let $k\geq 3$ be fixed. We partition the vertices of $\Basket_k$ into four disjoint sets as follows. First, let $\mathcal{A}_1 = \{v_1, \dotsc, v_k\}$ contain the $k$ vertices of the path $\Path_k$ in $\Basket_k$ including the endpoints of degree four. Then, for $i=2, 3, 4$, let $\mathcal{A}_i$ contain the $k-1$ interior vertices of each disjoint path $\Path_{k+1}$.  Based on this partition, the distance matrix $\distmatrix(\Basket_k)=:\distmatrix_k$ admits the following block decomposition:
            \begin{align*}
                \distmatrix_k &= \begin{bmatrix}
                    {P}_{k} & {B} & {B} & {B}\\
                    {B}^\top & {P}_{k-1} & {C} & {C}\\
                    {B}^\top & {C}^\top & {P}_{k-1} & {C}\\
                    {B}^\top & {C}^\top & {C}^\top & {P}_{k-1}
                \end{bmatrix}.
            \end{align*}
        Here, ${P}_m\in\mathbb{R}^{m\times m}$ denotes the distance matrix of the path graph $\Path_m$, ${B}\in\mathbb{R}^{k\times (k-1)}$ contains the distances between the shorter path and the interior vertices of each longer path, and finally ${C}\in\mathbb{R}^{(k-1)\times (k-1)}$ contains the distances between interior vertices of distinct longer paths.

        To construct the vector $\x $ satisfying $\one^\top\x  = 1$ and $\distmatrix_k \x\propto \one $, we proceed as follows. For $m\geq 1$, let $\vec{a}^{(m)}\in\mathbb{R}^{m}$ be given by the equation
            \begin{align}\label{eq:defn-am}
                \vec{a}^{(m)} &= \frac{1}{2} \begin{bmatrix}
                    1\\
                    {-3}\\
                    \vdots\\
                    ({-3})^{m-2}\\
                    ({-3})^{m-1}
                \end{bmatrix} + \frac{1}{2} \begin{bmatrix}
                    ({-3})^{m-1}\\
                    ({-3})^{m-2}\\
                    \vdots\\
                    {-3}\\
                    1
                \end{bmatrix},\quad \vec{a}^{(m)}_i &= \frac{1}{2}\left((-3)^{i-1} + (-3)^{m-i}\right),\quad 1\leq i \leq m.
            \end{align}
        Next, set
            \begin{align*}
                \x  &= \begin{bmatrix}
                    \vec{a}^{(k)}\\
                    \vec{a}^{(k-1)}\\
                    \vec{a}^{(k-1)}\\
                    \vec{a}^{(k-1)}
                \end{bmatrix}.
            \end{align*}

        First we observe that
            \begin{align*}
                \one^\top \x  &= \sum_{i=1}^{k} \vec{a}^{(k)}_i + 3\sum_{i=1}^{k-1} \vec{a}^{(k-1)}_i\\
                &= \frac{1}{2}\sum_{i=1}^{k} \left((-3)^{i-1} + (-3)^{k-i}\right) + \frac{3}{2}\sum_{i=1}^{k-1} \left((-3)^{i-1} + (-3)^{k-1-i}\right)\\
                &= \frac{1}{2}\left(\frac{1 - (-3)^{k}}{1 - (-3)} + \frac{1 - (-3)^{k}}{1 - (-3)}\right) + \frac{3}{2}\left(\frac{1 - (-3)^{k-1}}{1 - (-3)} + \frac{1 - (-3)^{k-1}}{1 - (-3)}\right)\\
                &= \frac{1 - (-3)^{k} + 3(1 - (-3)^{k-1})}{1 - (-3)} = 1,
            \end{align*}
        which proves the first half of~\cref{eq:alpha-k}. Next we claim $\distmatrix_k \x  \propto \one $. To see this, first we compute in block form the equation
            \begin{align}\label{eq:block-form-Dy}
                \distmatrix_k\x  &= \begin{bmatrix}
                    P_k\vec{a}^{(k)} + 3{B}\vec{a}^{(k-1)}\\
                    {B}^\top \vec{a}^{(k)} + {P}_{k-1}\vec{a}^{(k-1)} + 2{C}\vec{a}^{(k-1)}\\
                    {B}^\top \vec{a}^{(k)} + {P}_{k-1}\vec{a}^{(k-1)} + 2{C}\vec{a}^{(k-1)}\\
                    {B}^\top \vec{a}^{(k)} + {P}_{k-1}\vec{a}^{(k-1)} + 2{C}\vec{a}^{(k-1)}
                \end{bmatrix}.
            \end{align}
        We consider each term in turn. First, we have that for $1\leq i\leq k-1$, it holds
            \begin{align*}
                (P_k\vec{a}^{(k)})_{i+1} - ({(P_{k})}\vec{a}^{(k)})_{i} &= \sum_{j=1}^{k} {(P_{k})}_{i+1, j} \vec{a}^{(k)}_{j} - \sum_{j=1}^{k} {(P_{k})}_{i, j} \vec{a}^{(k)}_{j}\\
                &= \sum_{j=1}^{k} |i+1 - j| \vec{a}^{(k)}_{j} - \sum_{j=1}^{k} |i - j| \vec{a}^{(k)}_{j}\\
                &= \sum_{j=1}^{i} ((i+1 - j) - (i - j)) \vec{a}^{(k)}_{j} + \sum_{j=i+1}^{k} ((j - (i+1)) - (j - i)) \vec{a}^{(k)}_{j}\\
            \end{align*}
        For $s, N\ge 1$, set
            \begin{align*}
                T_{s, N} &= \sum_{i=1}^{N} \vec{a}^{(s)}_i.
            \end{align*}
        Thus,
            \begin{align}\label{eq:onesteppath}
                \sum_{j=1}^{i}& ((i+1 - j) - (i - j)) \vec{a}^{(k)}_{j} + \sum_{j=i+1}^{k} ((j - (i+1)) - (j - i)) \vec{a}^{(k)}_{j}\notag\\
                &= \sum_{j=1}^{i} \vec{a}^{(k)}_{j} - \sum_{j=i+1}^{k} \vec{a}^{(k)}_{j}\notag\\
                &= T_{k, i} - (T_{k, k} - T_{k, i}) = 2 T_{k, i} - T_{k, k}.
            \end{align}
        Separately, again for $1\leq i\leq k$ fixed, we have
            \begin{align*}
                ({B}\vec{a}^{(k-1)})_{i} &= \sum_{j=1}^{k-1} {B}_{i, j} \vec{a}^{(k-1)}_{j}\\
                &= \sum_{j=1}^{k-1} \min\{ i -1 +j , 2k-i-j \} \vec{a}^{(k-1)}_{j}.
            \end{align*}
        Therefore for each $1\leq i \leq k-1$, it holds
            \begin{align*}
                ({B}\vec{a}^{(k-1)})_{i+1} - ({B}\vec{a}^{(k-1)})_{i} &= \sum_{j=1}^{k-i-1} \vec{a}^{(k-1)}_{j} - \sum_{j=k - i+1}^{k-1} \vec{a}^{(k-1)}_{j}\\
                &= T_{k-1, k -i- 1} - T_{k-1, k-1} - T_{k-1, k-i}.
            \end{align*}    
        Next write, with $r=-3$ for shorthand,
            \begin{align}\label{eq:tki-exp}
                T_{k, i} &= \frac{1-r^{i} + r^{k-i} - r^{k}}{2(1-r)},
            \end{align}
        so that
            \begin{align}\label{eq:calc1}
                2 T_{k, i} - T_{k, k} &= \frac{1 - r^{i} + r^{k-i} - r^{k} - (1 - r^{k})}{1-r} = \frac{- r^{i} + r^{k-i}}{1-r},
            \end{align}
        and
            \begin{align}\label{eq:calc2}
                T_{k-1, k -i- 1} - T_{k-1, k-1} - T_{k-1, k-i} &= \frac{-r^{k-i-1} - r^{k-i} + r^{i} + r^{i-1}}{2(1-r)}.
            \end{align}
        Combining~\cref{eq:calc1} and~\cref{eq:calc2}, we have
            \begin{align*}
                \frac{- r^{i} + r^{k-i}}{1-r} + \frac{-r^{k-i-1} - r^{k-i} + r^{i} + r^{i-1}}{2(1-r)} &= \frac{r^{k-i} - r^{i}}{1-r} + \frac{3(r^{i} + r^{i-1} - r^{k-i} - r^{k-i-1} )}{2(1-r)}\\
                &= \frac{-r^{k-i} - 3r^{k-i-1} +r^{i} +3r^{i-1}}{2(1-r)}\\
                &=\frac{(r+3)(r^{i-1} - r^{k-i-1})}{2(1-r)} = 0,
            \end{align*}
        since $r+3 = 0$. It follows that for each $1\leq i \leq k-1$,
            \begin{align*}
                (P_k\vec{a}^{(k)})_{i+1} + ({B}\vec{a}^{(k-1)})_{i+1} - \left((P_k\vec{a}^{(k)})_{i}  + ({B}\vec{a}^{(k-1)})_{i}\right) = 0,
            \end{align*}
        so that the first component of~\cref{eq:block-form-Dy} is constant. Next we consider ${B}^\top \vec{a}^{(k)} + {P}_{k-1}\vec{a}^{(k-1)} + 2{C}\vec{a}^{(k-1)}$. Again fixing $1\leq i \leq k-2$, it follows by~\cref{eq:onesteppath} that
            \begin{align*}
                ({P}_{k-1}\vec{a}^{(k-1)})_{i+1} - ({P}_{k-1}\vec{a}^{(k-1)})_{i} &= 2 T_{k-1, i} - T_{k-1, k-1}.
            \end{align*}
        On the other hand, by a similar argument as before, we have
            \begin{align*}
                ({B}^\top \vec{a}^{(k)})_{i+1} - ({B}^\top \vec{a}^{(k)})_{i} &= \sum_{j=1}^{k} \min\{ i + j , 2k - i -1 - j \} \vec{a}^{(k)}_{j}\\
                &\qquad - \sum_{j=1}^{k} \min\{ i -1 + j , 2k - i - j \} \vec{a}^{(k)}_{j}\\
                &= T_{k, k-i-1} - (T_{k, k} - T_{k, k-i}).
            \end{align*}
        Separately, we have that
            \begin{align*}
                ({C}\vec{a}^{(k-1)})_{i+1} - ({C}\vec{a}^{(k-1)})_{i} &= \sum_{j=1}^{k-1}({C}_{i + 1, j} - {C}_{i, j})\vec{a}^{(k-1)}_{j}\\
                &= \sum_{j=1}^{k-i-1} \vec{a}^{(k-1)}_{j} - \sum_{j=k-i}^{k-1} \vec{a}^{(k-1)}_{j}\\
                &= 2T_{k-1,k-i-1} -T_{k-1, k-1}.
            \end{align*}
        Again using~\cref{eq:tki-exp}, we have
            \begin{align}\label{eq:exp-part2}
                2 T_{k-1, i} - T_{k-1, k-1} &= \frac{- r^{i} + r^{k-1-i}}{1-r},\notag\\
                T_{k, k-i-1} - (T_{k, k} - T_{k, k-i}) &= \frac{-r^{k-i-1} - r^{k-i} + r^{i+1} + r^{i}}{ 2(1-r)},\text{ and }\notag\\
                2T_{k-1,k-i-1} -T_{k-1, k-1} &= \frac{r^{i} - r^{k-i-1}}{1-r}.
            \end{align}
        Putting the parts of~\cref{eq:exp-part2} together, we have
            \begin{align*}
                2 T_{k-1, i} &- T_{k-1, k-1} + T_{k, k-i-1} - (T_{k, k} - T_{k, k-i}) + 2(2T_{k-1,k-i-1} -T_{k-1, k-1}) \\
                &= \frac{- r^{i} + r^{k-1-i}}{1-r} + \frac{-r^{k-i-1} - r^{k-i} + r^{i+1} + r^{i}}{ 2(1-r)} + \frac{2(r^{i} - r^{k-i-1})}{1-r}\\
                &= \frac{r^{i} - r^{k-1-i} }{1-r} + \frac{-r^{k-i-1} - r^{k-i} + r^{i+1} + r^{i}}{ 2(1-r)} \\
                &= \frac{-r^{k-i-1} - 3r^{k-i} + r^{i+1} + 3r^{i}}{2(1-r)}\\
                &= \frac{(r+3)(r^{i} - r^{k-i-1})}{2(1-r)} = 0,
            \end{align*}
        again since $r+3=0$. Therefore the remaining three blocks of~\cref{eq:block-form-Dy} are constant. We have now shown that the first block and the following three agree, and we proceed to compute their common value. First set, for $s, N\ge 1$,
            \begin{align*}
                S_{s, N} &= \sum_{i=1}^{N} i \vec{a}^{(s)}_i.
            \end{align*}
        By definition of $\vec{a}^{(s)}$ and basic facts of geometric sums, we have
            \begin{align*}
                S_{s, N} &= \frac{1}{2} \sum_{i=1}^{N} i r^{i-1} + \frac{1}{2} \sum_{i=1}^{N} i r^{s - i}\\
                &=\frac{1}{2(1-r)^2}\Big(1-(N+1)r^{N}+N r^{N+1}+ r^{s+1}-(N+1)r^{\,s+1-N}+N r^{\,s-N}\Big)\\
                &=\frac{1-(4N+1)(-3)^N + (-3)^{\,s+1} + (4N+3)(-3)^{\,s-N}}{32}.
            \end{align*}
        Using~\cref{eq:tki-exp} and the preceding, we have, after some simplifications,
            \begin{align*}
                (P_{k}\vec{a}^{(k)})_{1} + 3({B}\vec{a}^{(k-1)})_{1} &= \sum_{j=1}^{k} (j-1) \vec{a}^{(k)}_{j} + 3\sum_{j=1}^{k-1} (j-1) \vec{a}^{(k-1)}_{j}\\
                &= S_{k, k} - T_{k, k} + 3(S_{k-1, k-1} - T_{k-1, k-1})\\
                &= \frac{(-3)^k + 4k -1}{8}.
            \end{align*}
        Lastly we verify
            \begin{align*}
                ({P}_{k-1}\vec{a}^{(k-1)})_{1} + {B}^\top \vec{a}^{(k)}_{1} + 2{C}\vec{a}^{(k-1)}_{1} &= \sum_{j=1}^{k-1} (j-1) \vec{a}^{(k-1)}_{j} + \sum_{j=1}^{k} (j-1) \vec{a}^{(k)}_{j} + 2\sum_{j=1}^{k-1} (j-1) \vec{a}^{(k-1)}_{j}\\
                &= 3(S_{k-1, k-1} - T_{k-1, k-1}) + (S_{k, k} - T_{k, k})\\
                &= \frac{(-3)^k + 4k -1}{8},
            \end{align*}
        which proves the second half of~\cref{eq:alpha-k}.
    \end{proof}

\end{document}